\documentclass[a4paper]{article}
\usepackage{amssymb,amsmath,amsthm}
\usepackage{fullpage}
\usepackage{hyperref}
\usepackage{graphicx}
\newcommand{\ie}{\emph{i.e.}}
\newcommand{\eg}{\emph{e.g.}}
\newcommand{\cf}{\emph{cf.}}
\newcommand{\Com}{\mathbb{C}}
\newcommand{\Real}{\mathbb{R}}
\newcommand{\Nat}{\mathbb{N}}
\newcommand{\Int}{\mathbb{Z}}
\newcommand{\Rat}{\mathbb{Q}}

\newcommand{\dist}{\mathop{\mathrm{dist}}\nolimits}

\newcommand{\Dom}{\mathsf{D}}

\newcommand{\half}{\mbox{$\frac{\pi}{2}$}}
\newcommand{\const}{\mathrm{const}}
\newcommand{\eps}{\varepsilon}
\newcommand{\sii}{L^2}
\newtheorem{Theorem}{Theorem}
\newtheorem{Proposition}{Proposition}
\newtheorem{Corollary}{Corollary}
\newtheorem{Lemma}{Lemma}

\theoremstyle{definition}
\newtheorem{Remark}{Remark}

\numberwithin{equation}{section}
\usepackage[normalem]{ulem}
\usepackage{color}
\definecolor{DarkBlue}{rgb}{0,0.1,0.7}
		
\newcommand\soutD{\bgroup\markoverwith
{\textcolor{DarkBlue}{\rule[.01ex]{2pt}{1pt}}}\ULon}
\newcommand{\Hm}[1]{\leavevmode{\marginpar{\tiny%
$\hbox to 0mm{\hspace*{-0.5mm}$\leftarrow$\hss}%
\vcenter{\vrule depth 0.1mm height 0.1mm width \the\marginparwidth}%
\hbox to
0mm{\hss$\rightarrow$\hspace*{-0.5mm}}$\\\relax\raggedright #1}}}

\begin{document}
%
\title{\textbf{\Large
Spectral analysis of the diffusion operator 
with random jumps from the boundary
}}
\author{Martin Kolb$^{a}$ \ and \ David Krej\v{c}i\v{r}\'ik$^{b}$}
\date{\small 
\emph{
\begin{quote}
\begin{itemize}
\item[$a)$] 
Department of Mathematics, Universit\"at Paderborn,
Warburger Str.~100, 33098 Paderborn, Germany;
kolb@math.uni-paderborn.de.%
\item[$b)$] 
Department of Theoretical Physics, Nuclear Physics Institute ASCR,
25068 \v{R}e\v{z}, Czech Republic;
krejcirik@ujf.cas.cz.%
\end{itemize}
\end{quote}
}
\smallskip
25 June 2020}

\begin{center}
ARTICLE \ \fbox{ Math. Z. 284 (2016), 877--900 }
\quad + \quad
CORRIGENDUM \ \fbox{ Math. Z. (2019) }
\medskip \\
{\small
\url{https://doi.org/10.1007/s00209-016-1677-y}
\qquad\qquad
\url{https://doi.org/10.1007/s00209-019-02377-8}
}
\end{center}

\vspace{-5ex}

{\let\newpage\relax\maketitle}

\begin{abstract}
\noindent
Using an operator-theoretic framework in a Hilbert-space setting,
we perform a detailed spectral analysis of the one-dimensional
Laplacian in a bounded interval, subject to specific non-self-adjoint
connected boundary conditions modelling a random jump from the boundary
to a point inside the interval. 
In accordance with previous works, we find that all the eigenvalues are real.
As the new results, we derive and analyse the adjoint operator, 
determine the geometric and algebraic multiplicities of the eigenvalues,
write down formulae for the eigenfunctions 
together with the generalised eigenfunctions
and study their basis properties.
It turns out that the latter heavily depend on 
whether the distance of the interior point
to the centre of the interval divided by the length of the interval
is rational or irrational.
Finally, we find a closed formula for the metric operator
that provides a similarity transform of the problem
to a self-adjoint operator.
\end{abstract}
%
%

\section{Introduction}\label{Sec.Intro}
%
In this paper we are interested in the non-self-adjoint 
eigenvalue problem
\begin{equation}\label{problem}
\left\{
\begin{aligned}
  &-\psi'' = \lambda \psi 
  \qquad \mbox{in} \qquad (-\half,\half) \,,
  \\
  &\psi(\pm\half) = \psi(\half a) 
  \,,
\end{aligned}
\right.
\end{equation}
with a real parameter $a \in (-1,1)$. 
The operator~$H$ associated with~\eqref{problem} 
is the generator of the following stochastic process:
\begin{enumerate}
\item Start a Brownian motion with quadratic variation equal to~$2$ 
in the interval $(-\half,\half)$ and wait until it hits one of 
the boundary points~$\pm\frac{\pi}{2}$. 
\item At the hitting time of $\pm\frac{\pi}{2}$ 
the Brownian particle gets restarted 
in an interior point $\frac{\pi}{2} a$ 
and repeats the process at the previous step. 
\end{enumerate}
This process is sometimes described as 
the Brownian motion on the figure eight~\cite{Grigorescu-Kang-2002}. 
The existence of such a process is in fact elementary 
and it can be constructed by piecing together Brownian motions 
in a rather direct way.
The problem~\eqref{problem} can be also understood 
as a spectral problem for a non-self-adjoint graph
with regular boundary conditions~\cite{HKS}.  

There are several obvious generalisations of the stochastic process. 
Firstly, instead of restarting the process at the fixed point $\frac{\pi}{2}a$, 
one could restart it according to a given probability 
distribution~$\mu$ on $(-\frac{\pi}{2},\frac{\pi}{2})$. 
Secondly, one can even take two different probability distributions~$\mu_-$ 
and~$\mu_+$ on $(-\frac{\pi}{2},\frac{\pi}{2})$ 
and restart the process according to~$\mu_{\pm}$ 
depending on whether the boundary point~$\pm\frac{\pi}{2}$ has been hit. 
This generalised process leads to the following analogue of~\eqref{problem}:
\begin{equation}\label{genproblem}
\left\{
\begin{aligned}
  &-\psi'' = \lambda \psi 
  \qquad \mbox{in} \qquad (-\half,\half) 
  \,,
  \\
  &\psi(\pm \half) = \int_{-\half}^{\half} \psi(x) \, \mu_\pm(dx) 
  \,.
\end{aligned}
\right.
\end{equation}

Despite its apparent simplicity, the process leads to several interesting results. 
First of all, it has been shown by Leung et al.\ in~\cite{Leung-Li-Rakesh-2008} 
that, even in the most general setting described above, 
the spectrum of the operator~$H^{\mu_-,\mu_+}$
associated with~\eqref{genproblem} is purely real, 
a property which cannot be typically expected
for non-selfadjoint operators. It has also been shown in \cite{Leung-Li-Rakesh-2008}, that the spectral gap of $H^{\mu_-,\mu_+}$ is always greater
than the first Dirichlet eigenvalue of the Laplacian 
in the interval $(-\frac{\pi}{2},\frac{\pi}{2})$. Furthermore, it has been shown analytically in~\cite{Leung-Li-Rakesh-2008} 
and probabilistically in~\cite{Kolb-Wuebker-2011} 
that in the case of $\mu_+=\mu_-$ the spectral gap of the spectrum  
of the generator~$H^{\mu_-,\mu_+}$
always coincides with the second Dirichlet eigenvalue of 
the Laplacian in the interval $(-\frac{\pi}{2},\frac{\pi}{2})$, 
independently of the specific choice of $\mu_+=\mu_-$. 

Thus it is fair to say that this family of non-selfadjoint differential operators 
exhibits rich spectral features. 
This is our starting point and we aim to further develop some of 
the spectral-theoretic properties of members of this family 
of non-self-adjoint differential operators. 

In this paper we are concerned with the most simple case~\eqref{problem} 
and investigate the associated operator~$H$ 
from a purely spectral-theoretic perspective 
and complement existing results
which mainly focused on the determination of eigenvalues 
or even only on the spectral gap. 
We investigate the spectrum of the operator~$H$ and its adjoint~$H^*$,
determine algebraic multiplicities of the eigenvalues 
and analyse the basis properties of the set of eigenfunctions. 
Due to the non-self-adjointness of the operator, 
it is not at all clear in which sense the eigenfunctions 
can be expected to be a basis of the associated Hilbert space. 
In these respects we further develop certain strands of research 
first developed in~\cite{Grigorescu-Kang-2002}, 
whose authors calculated among other things 
the spectrum of the above operator in the case $a=0$;
see also \cite{Ben-Ari-Pinsky_2007} and \cite{Ben-Ari-Pinsky_2008}, 
where the authors derive 
results on the spectrum of the above operator 
including geometric multiplicities of the eigenvalues. 

The organisation of this paper is as follows.
In Section~\ref{Sec.setting} we properly define~$H$
as a closed operator in the Hilbert space $\sii((-\half,\half))$
and state its basic properties.
We also provide an \emph{a priori} proof of the reality
of the eigenvalues of~$H$, without the need to compute
the eigenvalues and eigenfunctions explicitly.
The latter is done only in Section~\ref{Sec.point},
where we analyse geometric degeneracies of the eigenvalues
(Proposition~\ref{Prop.H}).
In Section~\ref{adjoint} we find the adjoint operator~$H^*$
and compute its spectrum (Proposition~\ref{Prop.H*}).
These results enable us in Section~\ref{Sec.algebraic}
to eventually determine algebraic degeneracies of the eigenvalues of~$H$
(Proposition~\ref{Prop.algebraic}).
It turns out that the eigenvalue degeneracies 
heavily depend on Diophantine properties of the parameter~$a$.
\begin{Theorem}\label{Thm.algebraic}
All the eigenvalues of~$H$ are algebraically simple 
if, and only if, $a \not\in \Rat$.
\end{Theorem}

In the second part of the paper, namely in Section~\ref{Sec.basis},
we study basis properties of~$H$.
Using the explicit knowledge of the resolvent kernel of~$H$
constructed in Section~\ref{Sec.resolvent}, 
we first show in Section~\ref{Sec.complete}
that the eigenfunctions together with the generalised eigenfunctions
form a complete set in $\sii((-\half,\half))$.
Then we study the minimal completeness and conditional-basis properties
in Sections~\ref{Sec.complete.minimal} and~\ref{Sec.conditional}, respectively.
These results can be summarised as follows. 
\begin{Theorem}\label{Thm.basis.intro}
\ 
\begin{enumerate}
\item
If $a \not\in \Rat$, then the eigenfunctions of~$H$
form a minimal complete set but not a conditional basis
in $\sii((-\half,\half))$.
\item
If $a \in \Rat$, 
then the eigenfunctions of~$H$ 
do not form a minimal complete set
in $\sii((-\half,\half))$.
\end{enumerate}
\end{Theorem} 

Finally, in Section~\ref{Sec.metric}, 
we are interested in the possibility of the quasi-self-adjointness relation
\begin{equation}\label{quasi.intro}
  H^* \Theta = \Theta H
  \,,
\end{equation}
where~$\Theta$ is a positive operator called a \emph{metric}. 
The concept of quasi-self-adjoint operators goes back
to a seminal paper of Dieudonn\'e~\cite{Dieudonne_1961}
and has been renewed recently 
in the context of quantum mechanics
with non-self-adjoint operators;
we refer to~\cite{KSTV} and~\cite[Chap.~5]{KS-book} 
for more details and references.
\begin{Theorem}\label{Thm.metric.intro}
Let $a \not\in \Rat$. 
The operator~$H$ satisfies the relation~\eqref{quasi.intro}
with the operator~$\Theta$ explicitly given by~\eqref{metric.ours}.
The latter is a positive, bounded and invertible operator
(the inverse is unbounded).
\end{Theorem} 
In view of this theorem,
the reality of the spectrum of~$H$ 
can be understood as a consequence
of a generalised similarity to a self-adjoint operator. 
We would like to emphasise that we have an explicit 
and particularly simple formula~\eqref{metric.ours} 
for the metric operator~$\Theta$.
There are not many non-self-adjoint models in the literature 
for which the metric operator can be constructed in a closed form,
\cf~\cite{KSZ} and references therein. 

We conclude the paper by Section~\ref{Sec.end}
where we suggest some open problems.

\section{An operator-theoretic setting
and basic properties}\label{Sec.setting}
%
We understand~\eqref{problem} as a spectral problem
for the operator~$H$ in $\sii((-\half,\half))$ defined by
\begin{equation}\label{operator}
  H\psi := -\psi''
  \,, \qquad
  \psi \in \Dom(H) := 
  \big\{
  \psi \in H^2((-\half,\half)) \ \big| \
  \psi(-\half) = \psi(\half a) = \psi(\half)
  \big\}
  \,.
\end{equation}
Note that the boundary values are well defined 
due to the embedding 
$
  H^2((-\half,\half))
  \hookrightarrow
  C^1([-\half,\half]) 
$.

Let us first state some basic properties of~$H$.
In the sequel, $\|\cdot\|$ and $(\cdot,\cdot)$ denote respectively
the norm and inner product (antilinear in the first argument)
of the Hilbert space $\sii((-\half,\half))$.
\begin{itemize}
\item
$H$ is \textbf{densely defined} because 
$
  C_0^\infty\big((-\half,\half)\setminus\{\half a\}\big) \subset \Dom(H)
$
and $C_0^\infty\big((-\half,\half)\setminus\{\half a\}\big)$ is dense
in 
$
  \sii\big((-\half,\half)\setminus\{\half a\}\big) 
  \simeq \sii((-\half,\half))
$.

\item
$H$ is \textbf{closed}, which can be directly shown as follows.
First of all, let us notice that there exists a positive constant~$C$
such that
\begin{equation}\label{H1}
  \forall \psi \in \Dom(H) 
  \,, \qquad
  \|\psi'\|^2 \leq C \, (\|\psi\|^2+\|\psi''\|^2)
  \,.
\end{equation}
Indeed, integrating by parts and using the boundary conditions, 
we find
\begin{align*}
  \|\psi'\|^2 
  &= (\psi,-\psi'') 
  + \bar{\psi}(\half a) \, \big[ \psi'(\half) - \psi'(-\half) \big]
  \\
  & = (\psi,-\psi'') 
  + \bar{\psi}(\half a) \, (1,\psi'')
  \\
  &\leq \|\psi\| \|\psi''\|
  + |\psi(\half a)| \, \sqrt{\pi} \, \|\psi''\|
  \,,
\end{align*}
where the last line is due to the Schwarz inequality.
At the same time, by quantifying the embedding 
$
  H^1((-\half,\half))
  \hookrightarrow
  C^0([-\half,\half]) 
$,
we have 
\begin{equation}\label{embedding}
  |\psi(x)|^2 \leq \frac{1}{\pi} \, \|\psi\|^2
  + 2 \, \|\psi\| \|\psi'\|
  \leq \left(\frac{1}{\pi}+\frac{1}{\epsilon}\right) \|\psi\|^2
  + \epsilon \, \|\psi'\|^2
\end{equation}
for every $\psi \in H^1((-\half,\half))$,
$x \in [-\half,\half]$ 
and any $\epsilon>0$. 
Putting these two inequalities together, 
we verify~\eqref{H1}.

Now, let $\{\psi_n\}_{n=1}^\infty \subset \Dom(H)$ be such that 
$\psi_n \to \psi$ and $-\psi_n'' \to \phi$ as $n \to \infty$.
Applying~\eqref{H1} to~$\psi_n$, 
we see that $\{\psi_n\}_{n=1}^\infty$ is a bounded 
sequence in $H^2((-\half,\half))$ 
and thus weakly converging in this space. 
Hence, $\psi \in H^2((-\half,\half))$ and $\phi = -\psi''$.
Applying~\eqref{H1} to~$\psi_n-\psi$,
we see that $\psi_n \to \psi$ strongly in $H^2((-\half,\half))$
as $n \to \infty$.
The preservation of the boundary conditions 
in the limit is ensured by the embedding inequality~\eqref{embedding}. 
\item
The \textbf{numerical range} of~$H$ covers the whole complex plane,
\ie,
$$
  \{(\psi,H\psi) \ | \ \psi\in\Dom(H), \ \|\psi\|=1\} 
  = \Com \,.
$$
To see it,
we employ the identity
\begin{equation*}
  (\psi,H\psi) 
  = \|\psi'\|^2 - 
  \bar{\psi}(\half) \, [\psi'(\half) - \psi'(-\half)] 
  \,,
\end{equation*}
which follows by integrating by parts 
and using the boundary conditions.
Let $\phi \in C^\infty(\Real)$ be an arbitrary complex-valued function 
such that its support is contained in $(\half a,\half]$.
We set
$$
  \psi_\eps(x) := 
  \begin{cases}
  1 + \sqrt{\eps} \, \phi\big(\half-\eps^{-1}(\half-x)\big)
  & \mbox{if} \quad
  x \in [\half b_\eps, \half] \,,
  \\
  1 & \mbox{otherwise},
  \end{cases}
$$
with $b_\eps := 1-\eps(1-a)$ and any $\eps \in (0,1]$.
Clearly, $\psi_\eps \in \Dom(H)$ 
and a straightforward calculation yields
$$
  (\psi_\eps,H\psi_\eps) 
  = \|\phi'\|^2 - \frac{\phi'(\half)}{\sqrt{\eps}} 
  \,, \qquad
  \|\psi_\eps\|^2 = \pi
  + 2 \, \eps^{3/2} \Re \int_{\Real} \phi(y) \, dy 
  + \eps^2 \int_{\Real} |\phi(y)|^2 \, dy
  \xrightarrow[\eps \to 0]{} \pi 
  \,.
$$
Since $\phi'(\half)$ is an arbitrary complex number,
the desired claim follows by sending~$\eps$ to zero
and recalling the convexity of the numerical range.
\item
$H$ has purely \textbf{real eigenvalues}.
This striking property can be shown \emph{a priori},
without solving the eigenvalue problem explicitly, as follows.
Multiplying the first equation in~\eqref{problem} by~$\psi'$,
we arrive at the first integral
\begin{equation}\label{first}
  -\psi'^2 - \lambda \psi^2 = \const
  \qquad \mbox{in} \qquad (-\half,\half) \,.
\end{equation}
Using the boundary conditions of~\eqref{problem}, 
we thus deduce that the derivative of any eigenfunction~$\psi$ of~$H$
satisfies 
\begin{equation}\label{bc.der}
  \psi'(-\half)^2 = \psi'(\half a)^2 = \psi'(\half)^2
  \,.
\end{equation}
We divide the analysis into two cases now.
\begin{enumerate}
\item
Let $\psi'(\half a) = \psi'(\half)$.
Then~$\psi$ is a solution of the problem $-\psi''=\lambda\psi$ 
in $(\half a,\half)$, subject to periodic boundary conditions 
$\psi(\half a) = \psi(\half)$ and $\psi'(\half a) = \psi'(\half)$.
This is a self-adjoint problem and thus $\lambda \in \Real$.
Actually, 
$$
  \lambda = \left( \frac{4m}{1-a} \right)^2
  \,, \qquad
  m \in \Nat
  \,.
$$
The same argument applies to the situation $\psi'(\half a) = \psi'(-\half)$,
where we find
$$
  \lambda = \left( \frac{4m}{1+a} \right)^2
  \,, \qquad
  m \in \Nat
  \,.
$$
In this paper we use the convention $0 \in \Nat$ 
and set $\Nat^* := \Nat \setminus \{0\}$.

\item
Let $\psi'(\half a) = -\psi'(\half)$.
If $\psi'(\half a) = \psi'(-\half)$, we are in the previous case
for which we already know that the eigenvalues are real.
We may thus assume $\psi'(\half a) = -\psi'(-\half)$ as well.
But then~$\psi$ is a solution of the problem $-\psi''=\lambda\psi$ 
in the whole interval $(-\half,\half)$, 
subject to periodic boundary conditions 
$\psi(-\half) = \psi(\half)$ and $\psi'(-\half) = \psi'(\half)$. 
This is again a self-adjoint problem and thus $\lambda \in \Real$.
Actually, 
$$
  \lambda = \left( 2m \right)^2
  \,, \qquad
  m \in \Nat
  \,.
$$
\end{enumerate}
The above analysis implies:
$$
  \sigma_\mathrm{p}(H) \subset
  \left\{ 
  \left(\frac{4m}{1-a}\right)^2, \
  \left(\frac{4m}{1+a}\right)^2, \
  \left(2m\right)^2
  \right\}_{m \in \Nat}
  \,.
$$
The opposite inclusion $\supset$ will follow from
an explicit solution of the spectral problem~\eqref{problem}
(alternatively, we could construct admissible eigenfunctions 
for~\eqref{problem} from the periodic solutions discussed above,
but this would be almost like solving~\eqref{problem} explicitly).
\end{itemize}

The fact that the total spectrum of~$H$ is real
will follow from the reality of the eigenvalues established here,
but only after we show that~$H$ has a \emph{purely discrete spectrum}. 
To see the latter, we remark that $\Dom(H)$ 
is a subset of $H^2((-\half,\half))$,
which is compactly embedded in $\sii((-\half,\half))$.
But we still need to show that the resolvent set of~$H$ is not empty,
in order to show that~$H$ is an operator with compact resolvent.
To this aim, we shall determine the adjoint of~$H$.
First, however, let us study the point spectrum of~$H$ in detail. 

\section{The point spectrum}\label{Sec.point}
%
In this section we compute the point spectrum of~$H$
by solving the eigenvalue problem~\eqref{problem} explicitly.
Set $\lambda=:k^2$.
The general solution of the differential equation in~\eqref{problem} 
reads
(including $\lambda=0$)
$$
  \psi(x) = A \sin(kx) + B \cos(kx)
  \,, \qquad
  A,B \in \Com
  \,.
$$
Subjecting this solution to the boundary conditions of~\eqref{problem},
we arrive at the homogeneous system
\begin{equation}\label{homo}
\begin{pmatrix}
  \sin(k\half) + \sin(k\half a) 
  & -\cos(k\half) + \cos(k\half a) 
  \\
  \sin(k\half) - \sin(k\half a) 
  & \cos(k\half) - \cos(k\half a) 
\end{pmatrix}
\begin{pmatrix}
  A
  \\
  B 
\end{pmatrix}
  =
\begin{pmatrix}
  0
  \\
  0 
\end{pmatrix}
  .
\end{equation}
Eigenfunctions of~\eqref{problem} correspond to non-trivial 
solutions of this system, which in turn are determined 
by the singularity condition
$$
\begin{vmatrix}
  \sin(k\half) + \sin(k\half a) 
  & -\cos(k\half) + \cos(k\half a) 
  \\
  \sin(k\half) - \sin(k\half a) 
  & \cos(k\half) - \cos(k\half a) 
\end{vmatrix}
  = -4 \sin\left(k \mbox{$\frac{\pi}{4}$} (1+a)\right)
  \sin\left(k \mbox{$\frac{\pi}{4}$} (1-a)\right)
  \sin\left(k \half\right) 
  = 0
  \,.
$$
Consequently,
\begin{equation}\label{spec}
  \sigma_\mathrm{p}(H) =
  \left\{ 
  \left(\frac{4m}{1-a}\right)^2, \
  \left(\frac{4m}{1+a}\right)^2, \
  \left(2m\right)^2
  \right\}_{m \in \Nat}
  \,.
\end{equation}

It will be convenient to introduce the notation
\begin{equation}\label{sigmas}
  \sigma_{\pm 1} := 
  \left\{ 
  \left(\frac{4m}{1 \pm a}\right)^2
  \right\}_{m \in \Nat^*}
  \,, \qquad
  \sigma_{0} := 
  \left\{ 
  \left(2m\right)^2
  \right\}_{m \in \Nat}
  \,,
\end{equation}
and refer to eigenvalues from $\sigma_{+1}$, $\sigma_{-1}$ and $\sigma_0$
as eigenvalues from the ``$+1$~class'', ``$-1$~class'' and ``$0$~class'', 
respectively.
Note that zero is excluded from~$\sigma_{\pm 1}$
and that the sets $\sigma_{+1}$, $\sigma_{-1}$ and $\sigma_0$
are not disjoint in general.
Dependence of the eigenvalues on the parameter~$a$
is depicted in Figure~\ref{Fig}.

Now we specify the eigenfunctions associated with the individual classes. 
To study the eigenfunctions corresponding to the classes~$\pm 1$, 
it is useful to rewrite~\eqref{homo} into the form
\begin{equation}\label{homo.bis}
\begin{pmatrix}
  \sin(k\mbox{$\frac{\pi}{4}(1+a)$}) \cos(k\mbox{$\frac{\pi}{4}(1-a)$})
  & \sin(k\mbox{$\frac{\pi}{4}(1+a)$}) \sin(k\mbox{$\frac{\pi}{4}(1-a)$})
  \\
  \sin(k\mbox{$\frac{\pi}{4}(1-a)$}) \cos(k\mbox{$\frac{\pi}{4}(1+a)$})
  & -\sin(k\mbox{$\frac{\pi}{4}(1+a)$}) \sin(k\mbox{$\frac{\pi}{4}(1-a)$})
\end{pmatrix}
\begin{pmatrix}
  A
  \\
  B 
\end{pmatrix}
  =
\begin{pmatrix}
  0
  \\
  0 
\end{pmatrix}
  \,.
\end{equation}
\begin{itemize}
\item
\fbox{$-1$~class eigenvalues}
That is, $k = \frac{4m}{1-a}$ with $m \in \Nat^*$.
In this case, the second equation of~\eqref{homo.bis} 
is automatically satisfied, while the first yields the condition
$$
  A \, \sin\left(m\pi \mbox{$\frac{1+a}{1-a}$}\right) = 0
  \,.
$$
There are two possibilities: 
\begin{enumerate}
\item
If $m \frac{1+a}{1-a} \not\in \Nat$ (generic situation),
then $A=0$ and the eigenfunction associated with~$k^2$ reads
\begin{equation}\label{ef.-1.generic}
  \psi(x) = B \, \cos\left(\frac{4m}{1-a} x\right)
  \,,
\end{equation}
with a normalisation constant $B \in \Com\setminus\{0\}$.
\item
If $m \frac{1+a}{1-a} \in \Nat$ (exceptional situation),
then there are two (independent) eigenfunctions 
\begin{equation}\label{ef.-1.exceptional}
  \psi_1(x) = A \, \sin\left(\frac{4m}{1-a} x\right)
  \,, \qquad
  \psi_2(x) = B \, \cos\left(\frac{4m}{1-a} x\right)
  \,,
\end{equation}
with normalisation constants $A,B \in \Com\setminus\{0\}$.
\end{enumerate}
\item
\fbox{$+1$~class eigenvalues}
That is, $k = \frac{4m}{1+a}$ with $m \in \Nat^*$.
Here the situation is reversed with respect to the previous one.
Now the first equation of~\eqref{homo.bis} 
is automatically satisfied, while the second yields the condition
$$
  A \, \sin\left(m\pi \mbox{$\frac{1-a}{1+a}$}\right) = 0
  \,.
$$
There are again two possibilities: 
\begin{enumerate}
\item
If $m \frac{1-a}{1+a} \not\in \Nat$ (generic situation),
then $A=0$ and the eigenfunction associated with~$k^2$ reads
\begin{equation}\label{ef.+1.generic}
  \psi(x) = B \, \cos\left(\frac{4m}{1+a} x\right)
  \,,
\end{equation}
with a normalisation factor $B \in \Com\setminus\{0\}$.
\item
If $m \frac{1-a}{1+a} \in \Nat$ (exceptional situation),
then there are two (independent) eigenfunctions 
\begin{equation}\label{ef.+1.exceptional}
  \psi_1(x) = A \, \sin\left(\frac{4m}{1+a} x\right)
  \,, \qquad
  \psi_2(x) = B \, \cos\left(\frac{4m}{1+a} x\right)
  \,,
\end{equation}
with normalisation constants $A,B \in \Com\setminus\{0\}$.
\end{enumerate}
\item
\fbox{$0$~class eigenvalues}
That is, $k = 2m$ with $m \in \Nat$.
In this case, the two equations of~\eqref{homo} reduce to one
\begin{equation}\label{reduce}
  A \sin(m\pi a) 
  = B \left[ \cos(m\pi)-\cos(m\pi a) \right]
  \,.
\end{equation}
There are several possibilities:
\begin{enumerate}
\item
If $m=0$ (zero eigenvalue), there is just one (constant) 
eigenfunction 
\begin{equation}\label{ef.0.zero}
  \psi(x) = B \in \Com \setminus \{0\}
  \,.
\end{equation}
\item
If and $m \not=0$ and $m a \not \in \Nat$ (generic situation),
then we express~$A$ as a function of~$B$ 
and the eigenfunction associated with~$k^2$ reads
\begin{equation}\label{ef.0.generic}
  \psi(x) = B \left[ 
  \cos\left(2m x\right) 
  + \frac{\cos(m\pi)-\cos(m\pi a)}{\sin(m\pi a)} 
  \sin\left(2m x\right)
  \right]
  \,,
\end{equation}
with a normalisation constant $B \in \Com\setminus\{0\}$.
\item
If and $m \not=0$ and $m a \in \Nat$ (exceptional situation),
then~\eqref{reduce} reads
$$
  0 = B \left[ \cos(m\pi)-\cos(m\pi a) \right]
  = -2 \, B \, \sin\left(\frac{m\pi(1+a)}{2}\right) 
  \sin\left(\frac{m\pi(1-a)}{2}\right)
$$
and we still distinguish two cases:
\begin{enumerate}
\item
If $m(1+a)$ is odd (which necessarily implies that $m(1-a)$ is odd as well),
then $B=0$ and there is just one eigenfunction
\begin{equation}\label{ef.0.odd}
  \psi(x) = A \, \sin\left(2m x\right)
  \,,
\end{equation}
with a normalisation constant $A \in \Com\setminus\{0\}$.
\item
If $m(1+a)$ is even (which necessarily implies that $m(1-a)$ is even as well),
there are two (independent) eigenfunctions
\begin{equation}\label{ef.0.even}
  \psi_1(x) = A \, \sin\left(2m x\right)
  \,, \qquad
  \psi_2(x) = B \, \cos\left(2m x\right)
  \,,
\end{equation}
with normalisation constants $A,B \in \Com\setminus\{0\}$.
\end{enumerate}
\end{enumerate}
\end{itemize}

The exceptional situations in the classes~$-1$, $+1$ and~$0$ are related.
First of all, note that
$m\frac{1+a}{1-a} \in \Nat$, $m \frac{1-a}{1+a} \in \Nat$ or $ma \in \Nat$ 
with some $m \in \Nat^*$ imply that $a$~is rational.
Conversely, let~$a$ be rational.
Then the sets $\sigma_{-1}$, $\sigma_{+1}$ and $\sigma_0$ are not disjoint.
Clearly, $\lambda = (\frac{4m_{-1}}{1-a})^2 \in \sigma_{-1}$ 
with some $m_{-1} \in \Nat^*$ 
such that $m_{-1}\frac{1+a}{1-a} \in \Nat$ 
if, and only if,
$\lambda = (\frac{4 m_{+1}}{1+a})^2 \in \sigma_{+1}$ 
with some $m_{+1} \in \Nat^*$ 
such that $m_{+1}\frac{1-a}{1+a} \in \Nat$. 
At the same time, if $\lambda = (\frac{4m_{\pm 1}}{1 \pm a})^2 \in \sigma_{\pm 1}$ 
with some $m_{\pm 1} \in \Nat^*$ 
such that $m_{\pm 1}\frac{1 \mp a}{1 \pm a} \in \Nat$,
then there exists $m_0 \in \Nat^*$ such that 
$\lambda = (2m_0)^2 \in \sigma_{0}$.
On the other hand, if $\lambda = (2m_0)^2 \in \sigma_0$ 
with some $m_0 \in \Nat^*$ such that $m_0 a \in \Nat$
and $m_0(1+a)$ is even 
(which necessarily implies that $m_0(1-a)$ is even as well),
then there exist $m_{\pm 1} \in \Nat^*$ such that 
$m_{\pm 1}\frac{1 \mp a}{1 \pm a} \in \Nat$ and
$\lambda = (\frac{4m_{\pm 1}}{1 \pm a})^2 \in \sigma_{\pm 1}$.
Hence, all the exceptional situations 
with two independent eigenfunctions coincide with the intersection
$\sigma_{-1} \cap \sigma_{+1} = \sigma_{-1} \cap \sigma_{+1} \cap \sigma_{0}$, 
which is infinite,
and the elements of the intersection correspond to eigenvalues 
of geometric multiplicity two.
However, $\sigma_{-1} \cap \sigma_{+1} \not= \sigma_{0}$;
in fact, 
$
  \sigma_{0} \setminus (\sigma_{-1} \cup \sigma_{+1})
$
also contains an infinite number of elements,
which correspond to geometrically simple eigenvalues.

On the other hand, if~$a$ is irrational,
then the sets $\sigma_{-1}$ $\sigma_{+1}$ and $\sigma_{0}$
are mutually disjoint
and each point in the spectrum is an eigenvalue 
of geometric multiplicity one.

Let us summarise the spectral properties into the following proposition.
\begin{Proposition}\label{Prop.H}
$\sigma_\mathrm{p}(H)=\sigma_{-1} \cup \sigma_{+1} \cup \sigma_0$,
where the sets $\sigma_{-1}$, $\sigma_{+1}$ and $\sigma_0$
are introduced in~\eqref{sigmas}.
\begin{enumerate}
\item
If $a \not\in \Rat$, 
then the sets $\sigma_{-1}$ $\sigma_{+1}$ and $\sigma_{0}$
are mutually disjoint
and each point of the point spectrum corresponds to an eigenvalue of~$H$
of geometric multiplicity one,
with the associated eigenfunction
\eqref{ef.-1.generic}, \eqref{ef.+1.generic}, 
\eqref{ef.0.generic} or~\eqref{ef.0.zero}.
\item
If $a \in \Rat$, then 
$  
  \sigma_{-1} \cap \sigma_{+1} =
  \sigma_{-1} \cap \sigma_{+1} \cap \sigma_0 \not= \varnothing
$.
Each point of $\sigma_{-1} \cap \sigma_{+1}$
corresponds to an eigenvalue of~$H$ of geometric multiplicity two,
with the associated eigenfunctions 
\eqref{ef.-1.exceptional} and \eqref{ef.+1.exceptional}.
Each point of 
$\sigma_\mathrm{p}(H) \setminus (\sigma_{-1} \cap \sigma_{+1})$
corresponds to an eigenvalue of geometric multiplicity one,
with the associated eigenfunction 
\eqref{ef.-1.generic}, \eqref{ef.+1.generic}, \eqref{ef.0.generic},
\eqref{ef.0.odd} or \eqref{ef.0.even} 
or \eqref{ef.0.zero}
(zero eigenvalue, associated with the constant function~\eqref{ef.0.zero},
is always geometrically simple).
\end{enumerate}
\end{Proposition}

It is expected that the geometrically doubly degenerate eigenvalues
in $\sigma_{-1} \cap \sigma_{+1} \cap \sigma_0$
will have algebraic multiplicity three.
Indeed, fix $a \in \Rat$ and consider a point 
$\lambda \in \sigma_{-1} \cap \sigma_{+1} \cap \sigma_0$.
That is, there exists $l,m,n \in \Nat$ such that
$$
  \lambda 
  = \left(\frac{4l}{1-a}\right)^2
  = \left(\frac{4m}{1+a}\right)^2
  = (2n)^2
  \,.
$$
Introducing a small perturbation~$a \mapsto a+\eps$,
the eigenvalue~$\lambda$ splits into three \emph{distinct} eigenvalues
of geometric multiplicity one,
$$
  \lambda_{-1}(\eps) := \left(\frac{4l}{1-a-\eps}\right)^2 \in \sigma_{-1}
  \,, \qquad
  \lambda_{+1}(\eps) := \left(\frac{4m}{1+a+\eps}\right)^2 \in \sigma_{+1}
  \,, \qquad
  \lambda_0(\eps) := (2n)^2 \in \sigma_0
  \,,
$$
corresponding to mutually linearly independent eigenfunctions.

\begin{figure}[h!]
\begin{center}
\includegraphics[width=1\textwidth]{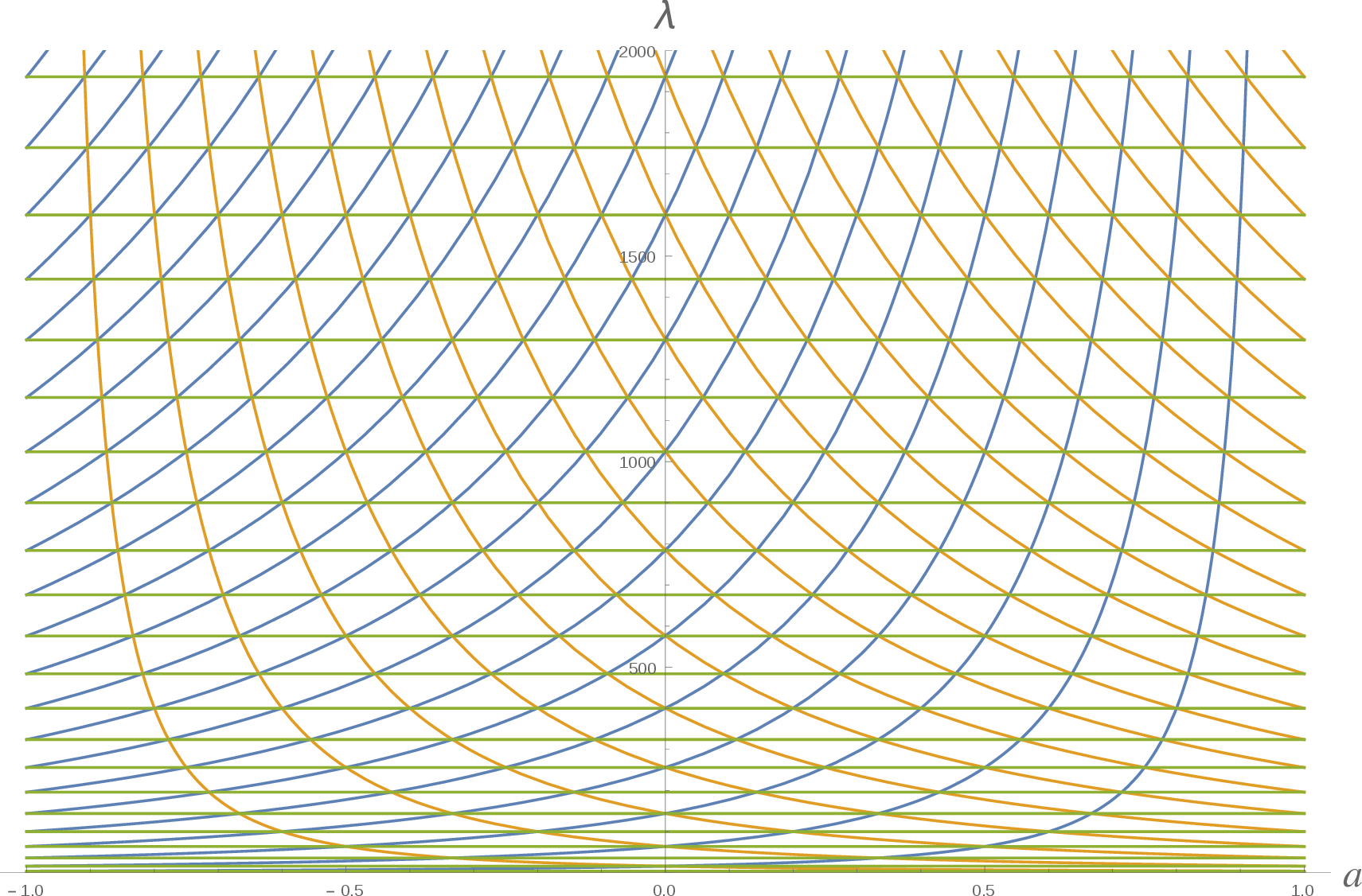}
\end{center}
\caption{Dependence of eigenvalues of~$H$ on~$a$.
The blue, yellow and green curves correspond to $-1$, $+1$ and $0$ class eigenvalues, 
respectively, \cf~\eqref{sigmas}. The multiplicities are clearly visible.}
\label{Fig}
\end{figure}

To discuss the algebraic degeneracies,
we first need to determine the adjoint of~$H$.

\section{The adjoint operator}\label{adjoint}
%
Obviously, $H$~is a closed extension of the symmetric operator
\begin{align*}
  (\dot{H}\psi)(x) &:= -\psi''(x)
  \,, \qquad 
  x \in (-\half,\half a) \cup (\half a,\half)
  \,, 
  \\
  \psi \in \Dom(\dot{H}) &:= 
  H_0^2\big((-\half,\half a)\big) \oplus H_0^2\big((\half a,\half)\big)
  \,.
\end{align*}
That is, $\dot{H} \subset H$.
The adjoint~$\dot{H}^*$ of~$\dot{H}$ is well known:
\begin{align*}
  (\dot{H}^*\psi)(x) &= -\psi''(x)
  \,, \qquad 
  x \in (-\half,\half a) \cup (\half a,\half)
  \,, 
  \\
  \psi \in \Dom(\dot{H}^*) 
  &= H^2\big((-\half,\half a)\big) \oplus H^2\big((\half a,\half)\big)
  \,.
\end{align*}
Since $\dot{H} \subset H \subset \dot{H}^*$, we also have
\begin{equation}\label{extension}
  \dot{H} \subset H^* \subset \dot{H}^*
  \,.  
\end{equation}

It follows that 
$\Dom(H^*) \subset H^2\big((-\half,\half a)\big) \oplus H^2\big((\half a,\half)\big)$
and that~$H^*$ acts as~$\dot{H}^*$.
Hence, we may integrate by parts to get the identity
\begin{align*}
  (\phi,H\psi) = (H^*\phi,\psi)
  &+ \psi(\half a) \left[
  \bar{\phi}'(\half a -) - \bar{\phi}'(\half a +)
  + \bar{\phi}'(\half) - \bar{\phi}'(-\half)
  \right]
  \\
  &+ \psi'(\half a) \left[
  \bar{\phi}(\half a +) - \bar{\phi}(\half a -)
  \right]
  \\
  &+ \psi'(-\half) \bar{\phi}(-\half) 
  - \psi'(\half) \bar{\phi}(\half) 
\end{align*}
for every $\psi \in \Dom(H)$ and $\phi \in \Dom(\dot{H}^*) \supset \Dom(H^*)$.
Using the arbitrariness of~$\psi$, we thus get
\begin{equation*}
\begin{aligned}
  (H^*\psi)(x) &= -\psi''(x)
  \,, \qquad 
  x \in (-\half,\half a) \cup (\half a,\half)
  \,, 
  \\
  \psi \in \Dom(H^*) 
  &= \left\{
  \psi \in H^2\big((-\half,\half a)\big) \oplus H^2\big((\half a,\half)\big)
  \ \left| \
  \begin{aligned}
    \phi(-\half) &= \phi(\half) = 0
    \\
    \phi(\half a -) &= \phi(\half a +) 
    \\
    \phi'(\half)-\phi'(-\half) &= \phi'(\half a +) - \phi'(\half a -)  
  \end{aligned}
  \right.
  \right\}
  .
\end{aligned}
\end{equation*}
Notice that $\Dom(H^*) \supset H_0^1((-\half,\half))$.

The point spectrum of~$H^*$ can be found by 
writing down the general solutions of $-\phi''=k^2\phi$
in $(-\half,\half a)$ and $(\half a,\half)$ 
and subjecting them to the boundary conditions of $\Dom(H^*)$.
Since the procedure is similar to our analysis for~$H$,
we just present the results.
We find that the eigenvalues of~$H$ and~$H^*$ coincide, \ie,
\begin{equation}\label{spec.equal}
  \sigma_\mathrm{p}(H^*) 
  = \sigma_\mathrm{p}(H)
  \,.  
\end{equation}
We again use the decomposition 
$\sigma_\mathrm{p}(H^*)=\sigma_{-1} \cup \sigma_{+1} \cup \sigma_0$
and specify the eigenfunctions associated with the individual classes.

\begin{itemize}
\item
\fbox{$-1$~class eigenvalues}
That is, $k = \frac{4m}{1-a}$ with $m \in \Nat^*$.
\begin{enumerate}
\item
If $m \frac{1+a}{1-a} \not\in \Nat$ (generic situation),
then the eigenfunction associated with~$k^2$ reads
\begin{equation}\label{ef*.-1.generic}
  \phi(x) = 
  \begin{pmatrix}
    0
    \\
    A_+ \, \sin\left(\frac{4m}{1-a} (x-\half)\right)
  \end{pmatrix}
\end{equation}
with a normalisation constant $A_+ \in \Com\setminus\{0\}$.
Here and in the sequel, for any 
$
  \phi = \phi_- \oplus \phi_+ \in 
  \sii((-\half,\half a)) \oplus \sii((\half a,\half))
$, 
we write
$
  \phi =
  \left(
  \begin{smallmatrix}
    \phi_-
    \\
    \phi_+
  \end{smallmatrix}
  \right)
$
.
\item
If $m \frac{1+a}{1-a} \in \Nat$ (exceptional situation),
then there are two (independent) eigenfunctions 
\begin{equation}\label{ef*.-1.exceptional}
  \phi_1(x) = 
  \begin{pmatrix}
    0
    \\
    A_+ \, \sin\left(\frac{4m}{1-a} (x-\half)\right)
  \end{pmatrix}
  \,, \qquad
  \phi_2(x) = 
    \begin{pmatrix}
    A_- \, \sin\left(\frac{4m}{1-a} (x+\half)\right)
    \\
    0
  \end{pmatrix}
  \,,
\end{equation}
with normalisation constants $A_\pm \in \Com\setminus\{0\}$.
\end{enumerate}
\item
\fbox{$+1$~class eigenvalues}
That is, $k = \frac{4m}{1+a}$ with $m \in \Nat^*$.
\begin{enumerate}
\item
If $m \frac{1-a}{1+a} \not\in \Nat$ (generic situation),
then the eigenfunction associated with~$k^2$ reads
\begin{equation}\label{ef*.+1.generic}
  \phi(x) = 
    \begin{pmatrix}
    A_- \, \sin\left(\frac{4m}{1+a} (x+\half)\right)
    \\
    0
  \end{pmatrix}
  \,,
\end{equation}
with a normalisation constant $A_- \in \Com\setminus\{0\}$.
\item
If $m \frac{1-a}{1+a} \in \Nat$ (exceptional situation),
then there are two (independent) eigenfunctions 
\begin{equation}\label{ef*.+1.exceptional}
  \phi_1(x) = 
  \begin{pmatrix}
    0
    \\
    A_+ \, \sin\left(\frac{4m}{1+a} (x-\half)\right)
  \end{pmatrix}
  \,, \qquad
  \phi_2(x) = 
    \begin{pmatrix}
    A_- \, \sin\left(\frac{4m}{1+a} (x+\half)\right)
    \\
    0
  \end{pmatrix}
  \,,
\end{equation}
with normalisation constants $A_\pm \in \Com\setminus\{0\}$.
\end{enumerate}
\item
\fbox{$0$~class eigenvalues}
That is, $k = 2m$ with $m \in \Nat$.
\begin{enumerate}
\item
If $m=0$ (zero eigenvalue), there is just one
eigenfunction 
\begin{equation}\label{ef*.0.zero}
  \phi(x)=
  \begin{pmatrix}
    C \, (a-1) (x+\half)
    \\
    C \, (a+1) (x-\half)
  \end{pmatrix}
  \,,
\end{equation}
with a normalisation constant $C \in \Com\setminus\{0\}$.
\item
If $m \not=0$ and $m a \not \in \Nat$ (generic situation),
the eigenfunction associated with~$k^2$ reads
\begin{equation}\label{ef*.0.generic}
  \phi(x) = 
  \begin{pmatrix}
    C \, \sin\left(2m (x+\half)\right)
    \\
    C \, \sin\left(2m (x-\half)\right)
  \end{pmatrix}
  \,,
\end{equation}
with a normalisation constant $C \in \Com\setminus\{0\}$.
\item
If $m \not=0$ and $m a \in \Nat$ (exceptional situation), 
we still distinguish two cases:
\begin{enumerate}
\item
If $m(1+a)$ is odd (which necessarily implies that $m(1-a)$ is odd as well),
there is just one eigenfunction,
which coincides with~\eqref{ef*.0.generic}.
\item
If $m(1+a)$ is even (which necessarily implies that $m(1-a)$ is even as well),
there are two (independent) eigenfunctions
\begin{equation}\label{ef*.0.even}
  \phi_1(x) = 
  \begin{pmatrix}
    0
    \\
    A_+ \, \sin\left(2m (x-\half)\right)
  \end{pmatrix}
  \,, \qquad
  \phi_2(x) = 
    \begin{pmatrix}
    A_- \, \sin\left(2m (x+\half)\right)
    \\
    0
  \end{pmatrix}
  \,,
\end{equation}
with normalisation constants $A_\pm \in \Com\setminus\{0\}$.
\end{enumerate}
\end{enumerate}
\end{itemize}

Let us summarise the spectral analysis of~$H^*$ into the following proposition. 

\begin{Proposition}\label{Prop.H*}
$\sigma_\mathrm{p}(H^*)=\sigma_{-1} \cup \sigma_{+1} \cup \sigma_0$,
where the sets $\sigma_{-1}$, $\sigma_{+1}$ and $\sigma_0$
are introduced in~\eqref{sigmas}.
\begin{enumerate}
\item
If $a \not\in \Rat$, then the sets
$\sigma_{-1}$, $\sigma_{+1}$ and $\sigma_0$
are mutually disjoint
and each point of the point spectrum corresponds to an eigenvalue of~$H^*$
of geometric multiplicity one,
with the associated eigenfunction
\eqref{ef*.-1.generic}, \eqref{ef*.+1.generic}, \eqref{ef*.0.generic}
or~\eqref{ef*.0.zero}.
\item
If $a \in \Rat$, then 
$
  \sigma_{-1} \cap \sigma_{+1} =
  \sigma_{-1} \cap \sigma_{+1} \cap \sigma_0 \not= \varnothing
$.
Each point of $\sigma_{-1} \cap \sigma_{+1}$
corresponds to an eigenvalue of~$H^*$ of geometric multiplicity two,
with the associated eigenfunctions 
\eqref{ef*.-1.exceptional} and \eqref{ef*.+1.exceptional}.
Each point of 
$\sigma_\mathrm{p}(H^*) \setminus (\sigma_{-1} \cap \sigma_{+1})$
corresponds to an eigenvalue of geometric multiplicity one,
with the associated eigenfunction 
\eqref{ef*.-1.generic}, \eqref{ef*.+1.generic}, \eqref{ef*.0.generic},
\eqref{ef*.0.even} or~\eqref{ef*.0.zero} 
(zero eigenvalue, associated with the function~\eqref{ef*.0.zero},
is always geometrically simple).
\end{enumerate}
\end{Proposition}

As the last result of this section, 
we show that~$H$ is an operator with compact resolvent.
\begin{Proposition}\label{Prop.compact}
$H$~is an operator with compact resolvent.
\end{Proposition}
\begin{proof}
Since $H^2((-\half,\half)) \supset \Dom(H)$ 
is compactly embedded in $\sii((-\half,\half))$,
it is enough to show that the resolvent of~$H$ exists 
at a point~$z$ of the complex plane
(and consequently at every point $z \not\in \sigma_\mathrm{p}(H)$).
We show it for $z=-1$, \ie, for every $F \in \sii((-\half,\half))$,
there exists $\psi \in \Dom(H)$ such that $(H+1)\psi=F$.
Indeed (using for instance the variation of constants),
the general solution of the differential equation $-\psi''+\psi=F$ reads 
$$
  \psi(x) = A e^x + B e^{-x} + g(x)
  \qquad \mbox{with} \qquad
  g(x) := -\int_0^x \cosh(x-x') \int_0^{x''} F(x'') \, dx'' \, dx'
  \,.
$$
Here~$A$ and~$B$ are complex constants to be determined
by the boundary conditions of~\eqref{operator}, that is,
$$
\begin{aligned}
  A (e^{\half}-e^{-\half}) + B (e^{-\half}-e^{\half})
  &= g(-\half) - g(\half) \,,
  \\ 
  A  (e^{\half}-e^{\half a}) + B (e^{-\half}-e^{-\half a})
  &= g(\half a) - g(-\half) \,.
\end{aligned}
$$
Since
$$
  (e^{\half}-e^{-\half})(e^{-\half}-e^{-\half a})
  - (e^{-\half}-e^{\half})(e^{\half}-e^{\half a})
  = 4 \sinh(\half) \left[\cosh(\half)-\cosh(\half a)\right] 
  >0
  \,,
$$
it is clear that the system above admits uniquely determined~$A$ and~$B$.
(Alternatively, we could use the explicit formula for the resolvent kernel
of~$H$  given in Proposition~\ref{p:resolventdecomp} below, 
whose statement remains valid for all 
$\lambda \in \Com\setminus [\sigma_\mathrm{p}(H)\cup\sigma_\mathrm{p}(H^0)]$.)
\end{proof} 
As a consequence of Proposition~\ref{Prop.compact},
the spectrum of~$H$ (as well as~$H^*$) is purely discrete,
in particular, it is exhausted by the eigenvalues~\eqref{spec}.
Summing up,
\begin{equation*}
  \sigma(H) = \sigma_{-1} \cup \sigma_{+1} \cup \sigma_0
  = \sigma(H^*)
  \,.
\end{equation*} 
%
%
\section{Algebraic multiplicities}\label{Sec.algebraic}
%
It is a general fact that $(\phi,\psi)=0$ is a necessary condition
for the existence of a generalised (root) vector for an eigenvalue~$\lambda$
of an operator~$H$, where~$\psi$ is a corresponding eigenfunction
and~$\phi$ is an eigenfunction of~$H^*$ corresponding to~$\bar\lambda$. 
The study of algebraic multiplicities of eigenvalues 
of our operator~$H$ is thus reduced 
to a computation of elementary trigonometric integrals. 

\begin{itemize}
\item
\fbox{$-1$~class eigenvalues}
Let $\lambda = \big(\frac{4m}{1-a}\big)^2$ with $m \in \Nat^*$.
\begin{enumerate}
\item
If $m \frac{1+a}{1-a} \not\in \Nat$ (generic situation),
we already know that the eigenvalue~$\lambda$ is geometrically simple.
The functions~$\psi$ and~$\phi$ are given 
by~\eqref{ef.-1.generic} and~\eqref{ef*.-1.generic}, respectively.
Since
\begin{equation}\label{norm.-1.generic}
  (\phi,\psi) = - \bar{A}_+ B \, \frac{\pi}{4} \, (1-a) \, 
  \sin\left(m\pi\,\frac{1+a}{1-a}\right) \cos(m \pi) \not= 0
  \,,
\end{equation} 
the eigenvalue~$\lambda$ is algebraically simple too.
\item
If $m \frac{1+a}{1-a} \in \Nat$ (exceptional situation),
we already know that the eigenvalue~$\lambda$
has geometric multiplicity two.
The two eigenfunctions~$\psi_1,\psi_2$ of~$H$
and the two eigenfunctions~$\phi_1,\phi_2$ of~$H^*$
are given by~\eqref{ef.-1.exceptional} and~\eqref{ef*.-1.exceptional}, 
respectively.
Since
\begin{equation}\label{norm.-1.exceptional}
\begin{aligned}
  (\phi_1,\psi_1) &= \bar{A}_+ A \, \frac{\pi}{4} \, (1-a) \,
  \cos\left(m\pi\,\frac{1+a}{1-a}\right) \cos(m \pi) \not= 0 
  \,,
  \\
  (\phi_2,\psi_1) &= \bar{A}_- A \, \frac{\pi}{4} \, (1+a) \,
  \cos\left(m\pi\,\frac{1+a}{1-a}\right) \cos(m \pi) \not= 0 
  \,,
  \\
  (\phi_1,\psi_2) &= 0 = (\phi_2,\psi_2)
  \,,
\end{aligned}
\end{equation} 
there might be a generalised eigenvector~$\xi$ of~$H$
associated with~$\psi_2$. 
In fact, the linearly independent solution of $(H-\lambda)\xi=\psi_2$
reads 
\begin{equation}\label{efg.-1.exceptional}
  \xi(x) := - B \, \frac{1-a}{64 m^2} \left[ 
  (1-a) \, \cos\left(\frac{4mx}{1-a}\right)
  + 8mx \, \sin\left(\frac{4mx}{1-a}\right)
  \right]
  \,.
\end{equation} 
Note that the function indeed belongs to~$\Dom(H)$
because necessarily $\frac{2m}{1-a} \in \Nat$,
\ie\ $\lambda \in \sigma_0$.
Hence, the algebraic multiplicity
of~$\lambda$ is at least three. 
To see that the algebraic multiplicity is not higher than three,
it is enough to verify that
\begin{equation}\label{norm.-1.generalised}
\begin{aligned}
  (\phi_1,\xi) 
  &= - \bar{A}_+ B \, \frac{\pi^2}{128 m} 
  \, (1-a)^2(1+a) \,
  \cos\left(m\pi\,\frac{1+a}{1-a}\right) \cos(m \pi) 
  \not= 0 
  \,,
  \\
  (\phi_2,\xi) 
  &= \bar{A}_- B \, \frac{\pi^2}{128 m} 
  \, (1-a)^2(1+a) \,
  \cos\left(m\pi\,\frac{1+a}{1-a}\right) \cos(m \pi) 
  \not= 0 
  \,.
\end{aligned}
\end{equation} 
\end{enumerate}
\item
\fbox{$+1$~class eigenvalues}
Let $\lambda = \big(\frac{4m}{1+a}\big)^2$ with $m \in \Nat^*$.
\begin{enumerate}
\item
If $m \frac{1-a}{1+a} \not\in \Nat$ (generic situation),
we already know that the eigenvalue~$\lambda$ is geometrically simple.
The functions~$\psi$ and~$\phi$ are given 
by~\eqref{ef.+1.generic} and~\eqref{ef*.+1.generic}, respectively.
Since
\begin{equation}\label{norm.+1.generic}
  (\phi,\psi) = \bar{A}_- B \, \frac{\pi}{4} \, (1+a) \, 
  \sin\left(m\pi\,\frac{1-a}{1+a}\right) \cos(m \pi) \not= 0
  \,,
\end{equation} 
the eigenvalue~$\lambda$ is algebraically simple too.
\item
If $m \frac{1-a}{1+a} \in \Nat$ (exceptional situation),
we already know that the eigenvalue~$\lambda$
has geometric multiplicity two.
The two eigenfunctions~$\psi_1,\psi_2$ of~$H$
and the two eigenfunctions~$\phi_1,\phi_2$ of~$H^*$
are given by~\eqref{ef.+1.exceptional} and~\eqref{ef*.+1.exceptional}, respectively.
Since
\begin{equation}\label{norm.+1.exceptional}
\begin{aligned}
  (\phi_1,\psi_1) &= \bar{A}_+ A \, \frac{\pi}{4} \, (1-a) \,
  \cos\left(m\pi\,\frac{1-a}{1+a}\right) \cos(m \pi) \not= 0 
  \,,
  \\
  (\phi_2,\psi_1) &= \bar{A}_- A \, \frac{\pi}{4} \, (1+a) \,
  \cos\left(m\pi\,\frac{1-a}{1+a}\right) \cos(m \pi) \not= 0 
  \,,
  \\
  (\phi_1,\psi_2) &= 0 = (\phi_2,\psi_2)
  \,,
\end{aligned}
\end{equation} 
there might be a generalised eigenvector~$\xi$ of~$H$
associated with~$\psi_2$. 
In fact, the linearly independent solution of $(H-\lambda)\xi=\psi_2$
reads
\begin{equation}\label{efg.+1.exceptional}
  \xi(x) := - B \, \frac{1+a}{64 m^2} \left[ 
  (1+a) \, \cos\left(\frac{4mx}{1+a}\right)
  + 8mx \, \sin\left(\frac{4mx}{1+a}\right)
  \right]
  \,.
\end{equation}
Note that the function indeed belongs to~$\Dom(H)$ 
because necessarily $\frac{2m}{1+a} \in \Nat$,
\ie\ $\lambda \in \sigma_0$. 
Hence, the algebraic multiplicity of~$\lambda$ is at least three. 
To see that the algebraic multiplicity is not higher than three,
it is enough to verify that
\begin{equation}\label{norm.+1.generalised}
\begin{aligned}
  (\phi_1,\xi) 
  &= - \bar{A}_+ B \, \frac{\pi^2}{128 m} 
  \, (1+a)^2(1-a) \,
  \cos\left(m\pi\,\frac{1-a}{1+a}\right) \cos(m \pi) 
  \not= 0 
  \,,
  \\
  (\phi_2,\xi) 
  &= \bar{A}_- B \, \frac{\pi^2}{128 m} 
  \, (1+a)^2(1-a) \,
  \cos\left(m\pi\,\frac{1-a}{1+a}\right) \cos(m \pi) 
  \not= 0 
  \,.
\end{aligned}
\end{equation}
\end{enumerate}
We remark that~\eqref{efg.+1.exceptional} can be deduced from~\eqref{efg.-1.exceptional}
by the replacement $m \mapsto m\frac{1-a}{1+a}$,
which reflects the relationship between the exceptional situations
in the~$+1$ and~$-1$ classes.
\item
\fbox{$0$~class eigenvalues}
Let $\lambda = (2m)^2$ with $m \in \Nat$.
\begin{enumerate}
\item 
If $m=0$, we already know that~$\lambda$ is geometrically simple.
The functions~$\psi$ and~$\phi$ are given 
by~\eqref{ef.0.zero} and~\eqref{ef*.0.zero}, respectively.
Since
\begin{equation}\label{norm.0.zero}
  (\phi,\psi) = - \bar{C} B \, \frac{\pi^2}{4} (1-a^2) \not= 0
  \,,
\end{equation}
the zero eigenvalue is always algebraically simple.
\item
If $m \not=0$ and $m a \not \in \Nat$ (generic situation),
we already know that the eigenvalue~$\lambda$ is geometrically simple.
The functions~$\psi$ and~$\phi$ are given 
by~\eqref{ef.0.generic} and~\eqref{ef*.0.generic}, respectively.
Since
\begin{equation}\label{norm.0.generic}
  (\phi,\psi) = \bar{C} B \, \frac{\pi}{2} \,
  \frac{1-\cos(m\pi)\cos(m\pi a)}{\sin(m\pi a)}
  \not= 0
  \,,
\end{equation}
the eigenvalue~$\lambda$ is algebraically simple too.
\item
If $m \not=0$ and $m a \in \Nat$ (exceptional situation),
we distinguish two cases:
\begin{enumerate}
\item
If $m(1+a)$ is odd (which necessarily implies that $m(1-a)$ is odd as well),
we already know that the eigenvalue~$\lambda$ is geometrically simple.
The eigenfunction~$\psi$ of~$H$ is given by~\eqref{ef.0.odd}
and the corresponding eigenfunction~$\phi$ of~$H^*$ is given by~\eqref{ef*.0.generic}.
Since
\begin{equation}\label{norm.0.odd}
  (\phi,\psi) = \bar{C} A \, \frac{\pi}{2} \cos(m\pi) 
  \not= 0
  \,,
\end{equation}
the eigenvalue~$\lambda$ is algebraically simple too.

\item
If $m(1+a)$ is even (which necessarily implies that $m(1-a)$ is even as well),
we already know that the eigenvalue~$\lambda$
has geometric multiplicity two.
The two eigenfunctions~$\psi_1,\psi_2$ of~$H$
and the two eigenfunctions~$\phi_1,\phi_2$ of~$H^*$
are given by~\eqref{ef.0.even} and~\eqref{ef*.0.even}, respectively.
Since
\begin{equation}\label{norm.0.even}
\begin{aligned}
  (\phi_1,\psi_1) &= \bar{A}_+ A \, \frac{\pi}{4} \, (1-a) \,
  \cos(m \pi) \not= 0 
  \,,
  \\
  (\phi_2,\psi_1) &= \bar{A}_- A \, \frac{\pi}{4} \, (1+a) \,
  \cos(m \pi) \not= 0 
  \,,
  \\
  (\phi_1,\psi_2) &= 0 = (\phi_2,\psi_2)
  \,,
\end{aligned}
\end{equation}
there might be a generalised eigenvector~$\xi$ of~$H$
associated with~$\psi_2$. 
In fact, the linearly independent solution of $(H-\lambda)\xi=\psi_2$
reads
\begin{equation}\label{efg.0}
  \xi(x) := - B \, \frac{1}{16 m^2} \left[ 
  \cos(2mx)
  + 4mx \, \sin(2mx)
  \right]
  \,.
\end{equation}
Hence, the algebraic multiplicity of~$\lambda$ is at least three.
To see that the algebraic multiplicity is not higher than three,
it is enough to verify that
\begin{equation}\label{norm.0.generalised}
\begin{aligned}
  (\phi_1,\xi) 
  &= - \bar{A}_+ B \, \frac{\pi}{64 m} 
  \, (1-a^2) \, \cos(m \pi) 
  \not= 0 
  \,,
  \\
  (\phi_2,\xi) 
  &= \bar{A}_- B \, \frac{\pi}{64 m} 
  \, (1-a^2) \, \cos(m \pi) 
  \not= 0 
  \,.
\end{aligned}
\end{equation}
We remark that~\eqref{efg.0} can be deduced from~\eqref{efg.-1.exceptional}
by the replacement $m \mapsto m\frac{1-a}{2}$,
which reflects the relationship between the exceptional situations
in the~$0$ and~$-1$ classes.
\end{enumerate}
\end{enumerate}
\end{itemize}

We summarise the established geometric and algebraic properties of
the eigenvalues of~$H$ in the following proposition.
\begin{Proposition}\label{Prop.algebraic}
\ 
\begin{enumerate}
\item
If $a \not\in \Rat$, then all the eigenvalues of~$H$
are algebraically simple.
\item
Let $a \in \Rat$.
Each point of 
$
  \sigma(H) \setminus (\sigma_{-1} \cap \sigma_{+1})
$
corresponds to an eigenvalue of~$H$ of algebraic multiplicity one.
Each point of 
$
  \sigma_{-1} \cap \sigma_{+1} 
  = \sigma_{-1} \cap \sigma_{+1} \cap \sigma_0
$ 
corresponds to an eigenvalue of~$H$ of geometric multiplicity two
and algebraic multiplicity three.
\end{enumerate}
\end{Proposition}

Theorem~\ref{Thm.algebraic} follows as a consequence of this proposition.

\section{The resolvent}\label{Sec.resolvent}
%
Now we turn to a study of the resolvent of~$H$ in some further detail. 
We have already seen in Section~\ref{adjoint} 
that the resolvent is a compact operator (\cf~Proposition~\ref{Prop.compact}). 
However, the compactness by itself is not sufficient 
to analyse completeness of eigenfunctions and related properties.
In this section we therefore give an explicit formula for the integral
kernel of the resolvent and show that it is a trace-class operator.

Let us denote by~$H^0$ 
the Laplacian in $\bigl(-\frac{\pi}{2},\frac{\pi}{2}\bigr)$
with Dirichlet boundary conditions, \ie,
\begin{equation*}
  H^0\psi := -\psi''
  \,, \qquad
  \psi \in \Dom(H^0) := 
  \big\{
  \psi \in H^2((-\half,\half)) \ \big| \
  \psi(-\half) = 0 = \psi(\half)
  \big\}
  \,,
\end{equation*}
and by $R^0(\lambda)$ its resolvent. 
It is well known that $\sigma(H^0)=\{n^2\}_{n\in\Nat^*}$
and that $R^0(\lambda)$ acts as an integral operator
with explicit kernel (see, \eg, \cite[Sec.~III.2.3]{Kato})
\begin{equation}\label{Green}
G_\lambda^0(x,y) := \frac{-1}{{k \sin(2k\half)}}
\begin{cases}
\sin(k(x+\half)) \sin(k(y-\half)) \,, & x<y \,, \\
\sin(k(y+\half)) \sin(k(x-\half)) \,, & x>y \,,
\end{cases}
\end{equation}
where~$k \in \Com$ is such that 
$k^2=\lambda \in \Com\setminus \sigma(H^0)$.

We have the following Krein-type formula
for the resolvent~$R(\lambda)$ of~$H$.  
\begin{Proposition}\label{p:resolventdecomp}
For every $\lambda \in \Com\setminus [\sigma(H)\cup\sigma(H^0)]$,
the resolvent $R(\lambda)$ of $H$ admits the decomposition
\begin{equation}\label{resolvent.decomposition}
  (R(\lambda)f)(x) 
  = (R^0(\lambda)f)(x) 
  + \frac{h^{x}(\lambda)}{1-h^{\frac{\pi}{2}a}(\lambda)} 
  \, (R^0(\lambda)f)(\half a)
  \,,
\end{equation} 
with any
$f \in \sii((-\half,\half))$ and $x \in [-\half,\half]$,
where
$$
  h^x(\lambda) := 
  \frac{\cosh(\sqrt{-\lambda} \, x)}{\cosh(\sqrt{-\lambda} \, \frac{\pi}{2})}
  \,.
$$
\end{Proposition}
\begin{proof}
First of all, notice that $R(\lambda)$ introduced by~\eqref{resolvent.decomposition}
is a bounded operator on $\sii((-\half,\half))$.
Indeed, it is the case of $R^0(\lambda)$ for $\lambda \in \Com\setminus \sigma(H^0)$
and the second term on the right hand side of~\eqref{resolvent.decomposition} 
represents a rank-one perturbation of~$R^0(\lambda)$.
More specifically, 
$$
  \frac{h^{x}(\lambda)}{1-h^{\frac{\pi}{2}a}(\lambda)} 
  \, (R^0(\lambda)f)(\half a)
  = g_1(x) \, (g_2,f)
  \,,
$$
where
$$
  g_1(x) := \frac{h^{x}(\lambda)}{1-h^{\frac{\pi}{2}a}(\lambda)}
  \qquad \mbox{and} \qquad
  g_2(y) := \overline{G_\lambda^0(\half a,y)} 
$$
are continuous functions on $[-\half,\half]$
for all $\lambda \in \Com\setminus [\sigma(H)\cup\sigma(H^0)]$.
Next, we observe that the function $x \mapsto (R(\lambda)f)(x)$ 
solves the boundary conditions
$$
  (R(\lambda)f)(-\half)=(R(\lambda)f)(\half a)=(R(\lambda)f)(\half)
  \,.
$$
Indeed, 
$$
  (R(\lambda)f)(-\half)
  =\frac{1}{1-h^{\frac{\pi}{2}a}(\lambda)}\,(R^0(\lambda)f)(\half a)
  =(R(\lambda)f)(\half)
$$
and
$$
  (R(\lambda)f)(\half a) 
  = (R^0(\lambda)f)(\half a)
  \biggl(1 + \frac{h^{\frac{\pi}{2} a}(\lambda)}
  {1-h^{\frac{\pi}{2}a}(\lambda)}\biggr)
  =\frac{1}{1-h^{\frac{\pi}{2}a}(\lambda)} \, (R^0(\lambda)f)(\half a)
  \,.
$$
Furthermore, it is straightforward to check that,
for every $f \in \sii((-\half,\half))$,
$R(\lambda)f \in H^2((-\half,\half))$ and
$$
  -(R(\lambda)f)''- \lambda \, (R(\lambda)f) = f
  \,. 
$$
Hence, $R(\lambda):\sii((-\half,\half)) \to \Dom(H)$ 
and $R(\lambda)$ is the right inverse of~$H-\lambda$.
To show that $R(\lambda)$ is also the left inverse of~$H-\lambda$,
one can employ~\eqref{Green}, 
which in particular yields the useful identity
$$
  [R^0(\lambda) (H-\lambda) \psi](x)
  = \psi(x) - \frac{\cos(kx)}{\cos(k\half)} \, \psi(\half a) 
$$
for every $\psi \in \Dom(H)$ and~$k \in \Com$ such that 
$k^2=\lambda \in \Com\setminus \sigma(H^0)$.
\end{proof}
\begin{Remark}
Formula~\eqref{resolvent.decomposition} can be deduced 
from \cite[Thm.~1]{Grigorescu-Kang-2002} 
(see also \cite[Eq.~(3.5)]{Grigorescu-Kang-2002}). 
However, since the transition semigroup of~\cite{Grigorescu-Kang-2002}
is defined on a different functional space, 
the present proof of Proposition~\ref{p:resolventdecomp}
is still needed.
\end{Remark}

From Proposition~\ref{p:resolventdecomp} 
we get the following corollary.
\begin{Proposition}\label{Prop.trace}
For every $\lambda \in \Com\setminus \sigma(H)$, 
the resolvent $R(\lambda)$ is a trace-class operator. 
\end{Proposition}
\begin{proof}
From Proposition~\ref{p:resolventdecomp} we see that
the resolvent~$R(\lambda)$ is a rank-one perturbation of~$R^0(\lambda)$. 
Since $R^0(\lambda)$ is well known to be trace-class,
rank-one operators are obviously trace-class
and trace-class operators form a two-sided ideal 
in the space of bounded operators
(see, \eg, \cite[Thm.~7.8]{Weidmann}),
we immediately obtain the claim from Proposition~\ref{p:resolventdecomp}
for every $\lambda \in \Com\setminus [\sigma(H)\cup\sigma(H^0)]$.
By the first resolvent identity~\cite[Thm.~5.13]{Weidmann}
and the two-sided ideal properties of trace-class operators,
the trace-class property then easily extends to all~$\lambda$
in the resolvent set of~$H$.  
\end{proof}
%

\section{Basis properties}\label{Sec.basis}
%
Since the spectrum of~$H$ is real, it is natural to ask 
whether~$H$ is similar to a self-adjoint operator.
This question is related to basis properties of
the eigenfunctions of~$H$.

\subsection{Completeness}\label{Sec.complete}
Recall that the \emph{completeness} of a family of vectors
$\{\psi_j\}_{j\in\Nat}$ in a Hilbert space~$\mathcal{H}$
means that its span is dense in~$\mathcal{H}$,
or equivalently, $(\{\psi_j\}_{j\in\Nat})^\bot = \{0\}$. 
\begin{Theorem}\label{Thm.complete}
The eigenfunctions of~$H$ together with the generalised eigenfunctions
form a complete set in $\sii((-\half,\half))$.
\end{Theorem} 
\begin{proof}
We use a completeness criterion due to Dunford and Schwartz
\cite[Corol.~XI.6.31]{Dunford-Schwartz_2},
which requires that the resolvent of~$H$ is a Hilbert-Schmidt operator 
and its norm admits an algebraic decay with respect 
to the spectral parameter along several rays in the complex plane.
The former follows from Proposition~\ref{Prop.trace}.
To establish the latter, we come back to Proposition~\ref{p:resolventdecomp}
and compute the Hilbert-Schmidt norm of~\eqref{resolvent.decomposition}
for $\lambda = -k^2$ with 
$
  k \in \mathcal{R}_c := \{k \in \Com \ |\ 
  \Im k = c \, \Re k \ \land \ \Re k \geq 1
  \}
$
where $c\in\Real$.
Notice that for any $k \in \mathcal{R}_c$ it holds
$|k| \to + \infty$ if, and only if, $\Re k \to + \infty$;
in fact, $|k|^2 = (1+c^2) (\Re k)^2$.
Since 
$$
  |h^x(\lambda)|^2 = 
  \frac{\cosh(2 \, \Re k \, x)+\cos(2 \, \Im k \, x)}
  {\cosh(2 \, \Re k \, \half)+\cos(2 \, \Im k \, \half)}
  \sim e^{-2 \Re k(\half-|x|)}
  \qquad \mbox{as} \qquad
  k \to \infty
  \quad \mbox{in} \quad \mathcal{R}_c
$$
and $h^{\frac{\pi}{2}a}(\lambda) = 1$ if, and only if, $\lambda=0$,
we have 
$$
  \sup_{k \in \mathcal{R}_c}
  \sup_{x\in(-\half,\half)}
  \left|
  \frac{h^{x}(\lambda)}{1-h^{\frac{\pi}{2}a}(\lambda)} 
  \right|
  < \infty 
  \,.
$$
Hence it is enough to compute the Hilbert-Schmidt norms
of $G_\lambda^0(x,y)$ and $G_\lambda^0(\half a,y)$
with $x,y \in (-\half,\half)$. 
Using
\begin{multline*}
|G_\lambda^0(x,y)|^2 = \frac{1}{2\,|k|^2 \, [\cosh(\Re k \,\pi)-\cos(\Im k \,\pi)]}
\\
\times
\begin{cases}
[\cosh(\Re k (x+\half))-\cos(\Im k (x+\half))]
[\cosh(\Re k (y-\half))-\cos(\Im k (y-\half))]
\,, 
& x<y \,, \\
[\cosh(\Re k (y+\half))-\cos(\Im k (y+\half))]
[\cosh(\Re k (x-\half))-\cos(\Im k (x-\half))]
\,, 
& x>y \,,
\end{cases}
\end{multline*}
an explicit calculation yields
\begin{align*}
  \lefteqn{\int_{-\half}^{\half} |G_\lambda^0(x,y)|^2 \, dy
  = \frac{1}{2\,|k|^2 \, [\cosh(\Re k \, \pi)-\cos(\Im k \, \pi)]}}
  \\
  &&&\times \Bigg\{
  \left[\frac{\sinh(\Re k (x+\half))}{\Re k}-\frac{\sin(\Im k (x+\half))}{\Im k}\right]
  [\cosh(\Re k (x-\half))-\cos(\Im k (x-\half))]
  \\
  &&&\qquad
  - [\cosh(\Re k (x+\half))-\cos(\Im k (x+\half))]
  \left[\frac{\sinh(\Re k (x-\half))}{\Re k}-\frac{\sin(\Im k (x-\half))}{\Im k}\right]
  \Bigg\}
  \,.
\end{align*}
Consequently,
$$
  \int_{-\half}^{\half} \int_{-\half}^{\half} |G_\lambda^0(\half a,y)|^2 \, dy \, dx
  \sim \frac{\pi}{2\,(1+c^2)\,(\Re k)^3} 
  \qquad \mbox{as} \qquad
  k \to \infty
  \quad \mbox{in} \quad \mathcal{R}_c
  \,.
$$
At the same time, an explicit calculation yields
\begin{align*}
  \lefteqn{
  \int_{-\half}^{\half} \int_{-\half}^{\half} |G_\lambda^0(x,y)|^2 \, dy \, dx
  }
  \\
  &&&= \frac{1}{2\,|k|^2 \, [\cosh(\Re k \, \pi)-\cos(\Im k \, \pi)]}
  \left\{
  \frac{4 \cos(\Im k \, \pi)}{|k|^2} - \frac{4 \cosh(\Re k \, \pi)}{|k|^2}  
  + \frac{\pi \sin(\Im k \, \pi)}{\Im k} + \frac{\pi \sin(\Re k \, \pi)}{\Re k}
  \right\}
  \\
  &&& \sim \frac{\pi}{(1+c^2) \, (\Re k)^3}
  \qquad \mbox{as} \qquad
  k \to \infty
  \quad \mbox{in} \quad \mathcal{R}_c
  \,.
\end{align*}
Summing up, for every $\theta \in (0,2\pi)$,
\begin{equation}\label{rays}
  \|R(\lambda)\| \leq 
  \|R(\lambda)\|_{\mathrm{HS}} = O(|\lambda|^{-3/2})
\end{equation}
as $\lambda \to \infty$
in the ray 
$
  \{\lambda \in \Com \ | \
  \sin\theta \, \Re\lambda - \cos\theta \, \Im\lambda = 0
  \ \land \ \cos\theta \, \Re\lambda > 0\}
$.
\end{proof}

As a consequence of this theorem and Proposition~\ref{Prop.algebraic},
we get
\begin{Corollary}\label{Corol.complete}
If $a \not\in \Rat$,
the eigenfunctions of~$H$ form a complete set in $\sii((-\half,\half))$.
\end{Corollary} 

Given $\lambda \in \Com\setminus[0,+\infty)$,
let~$G_\lambda$ and $G_\lambda^*$ denote the integral kernels 
of $R(\lambda) = (H-\lambda)^{-1}$ and $(H^*-\lambda)^{-1}$, respectively. 
Since $(H^*-\bar{\lambda})^{-1} = [(H-\lambda)^{-1}]^*$,
we have
$
  G_{\bar{\lambda}}^*(x,y) = \overline{G_\lambda(y,x)}
$
for every $x,y \in (-\half,\half)$.
Using additionally the symmetry $G_\lambda^0(x,y) = G_\lambda^0(y,x)$,
the previous proof immediately gives the same result~\eqref{rays}
for $(H^*-\lambda)^{-1}$.
Consequently, the statements of 
Theorem~\ref{Thm.complete} and Corollary~\ref{Corol.complete}
apply to the eigensystem of~$H^*$ as well.

\subsection{Minimal completeness}\label{Sec.complete.minimal}
We say that a complete set of vectors $\{\psi_j\}_{j\in\Nat}$
in a Hilbert space~$\mathcal{H}$ is \emph{minimal complete}
if the removal of any term makes it incomplete.
By~\cite[Prob.~3.3.2]{Davies_2007}, $\{\psi_j\}_{j\in\Nat}$
is minimal complete if, and only if, there exists 
a sequence $\{\phi_j\}_{j\in\Nat} \subset \mathcal{H}$
such that the pair is \emph{biorthogonal}, \ie,
\begin{equation}\label{bi}
  (\phi_j,\psi_k) = \delta_{jk}
\end{equation} 
for all $j,k \in \Nat$.

In our case, we form $\{\psi_j\}_{j\in\Nat}$ from
the eigenfunctions~$\psi$ of~$H$ together with the generalised
eigenfunctions~$\xi$.  
The dual sequence $\{\phi_j\}_{j\in\Nat}$ will be then given by
the eigenfunctions~$\phi$ of~$H^*$ together with its generalised
eigenfunctions~$\eta$ that we determine only now.  

\begin{itemize}
\item
\fbox{$-1$~class eigenvalues}
Let $\lambda = \big(\frac{4m}{1-a}\big)^2$ with $m \in \Nat^*$.
\begin{enumerate}
\item
If $m \frac{1+a}{1-a} \not\in \Nat$ (generic situation),
the eigenvalue~$\lambda$ is algebraically simple.
In view of~\eqref{norm.-1.generic}, 
the functions~$\psi$ and~$\phi$ given 
by~\eqref{ef.-1.generic} and~\eqref{ef*.-1.generic}, respectively,
can be normalised in such a way that~\eqref{bi} holds.
\item
If $m \frac{1+a}{1-a} \in \Nat$ (exceptional situation),
the eigenvalue~$\lambda$ has geometric multiplicity two
and algebraic multiplicity three.
In view of~\eqref{norm.-1.exceptional} and~\eqref{norm.-1.generalised},
the functions~$\psi_1, \xi$ given 
by~\eqref{ef.-1.exceptional} and~\eqref{efg.-1.exceptional}
and the functions $\phi_1, \phi_2$ given by~\eqref{ef*.-1.exceptional}
are mutually biorthogonal when normalised properly.
We still need to find the function dual to~$\psi_2$ from~\eqref{ef.-1.exceptional}. 
To this aim, we consider the equation $(H^*-\lambda)\eta=\phi_1+\phi_2$
and find the linearly independent solution   
\begin{equation}\label{efg*.-1.exceptional}
  \eta(x) := 
  \begin{pmatrix}
    A_- \frac{1-a}{64m^2} \left[
    8 m (x+\frac{\pi}{2}) \cos\left(\frac{4m}{1-a}(x+\frac{\pi}{2})\right)
    - (1-a) \sin\left(\frac{4m}{1-a}(x+\frac{\pi}{2})\right)
    \right]
    \smallskip \\
    A_+ \frac{1-a}{64m^2} \left[
    8 m (x-\frac{\pi}{2}) \cos\left(\frac{4m}{1-a}(x-\frac{\pi}{2})\right)
    - (1-a) \sin\left(\frac{4m}{1-a}(x-\frac{\pi}{2})\right)
    \right]
  \end{pmatrix}
  \,,
\end{equation} 
which indeed belongs to~$\Dom(H^*)$ provided that
\begin{equation}\label{efg*.-1.condition}
  A_- (1+a) = - A_+ (1-a)
  \,,
\end{equation} 
where~$A_\pm$ are the normalisation constants from~\eqref{ef*.-1.exceptional}.
Since
\begin{equation}\label{norm*.-1.exceptional}
  (\eta,\psi_2) =
  \bar{A}_- B \, \frac{\pi^2}{64 m} \, (1-a) \, (1+a) \, 
  \cos\left(m\pi\,\frac{1+a}{1-a}\right) \cos(m \pi)
  \not= 0 
  \,,
\end{equation} 
we can eventually choose the normalisation constants in such a way that
$\psi_2$ and~$\eta$ is the remaining biorthogonal pair. 
\end{enumerate}
\item
\fbox{$+1$~class eigenvalues}
Let $\lambda = \big(\frac{4m}{1+a}\big)^2$ with $m \in \Nat^*$.
\begin{enumerate}
\item
If $m \frac{1-a}{1+a} \not\in \Nat$ (generic situation),
the eigenvalue~$\lambda$ is algebraically simple.
In view of~\eqref{norm.+1.generic}, 
the functions~$\psi$ and~$\phi$ given 
by~\eqref{ef.+1.generic} and~\eqref{ef*.+1.generic}, respectively,
can be normalised in such a way that~\eqref{bi} holds.
\item
If $m \frac{1-a}{1+a} \in \Nat$ (exceptional situation),
then~$\lambda$ belongs to the exceptional situation in the~$-1$ class too.
Hence, the analysis is reduced to the preceding case.
In particular, the formula~\eqref{efg*.-1.exceptional}
holds here after the replacement
$m \mapsto m \frac{1-a}{1+a}$. 
\end{enumerate}
\item
\fbox{$0$~class eigenvalues}
Let $\lambda = (2m)^2$ with $m \in \Nat$.
\begin{enumerate}
\item 
If $m=0$, the eigenvalue~$\lambda$ is algebraically simple.
In view of~\eqref{norm.0.zero}, 
the functions~$\psi$ and~$\phi$ given 
by~\eqref{ef.0.zero} and~\eqref{ef*.0.zero}, respectively,
can be normalised in such a way that~\eqref{bi} holds.
\item
If $m \not=0$ and $m a \not \in \Nat$ (generic situation),
the eigenvalue~$\lambda$ is algebraically simple.
The functions~$\psi$ and~$\phi$ given 
by~\eqref{ef.0.generic} and~\eqref{ef*.0.generic}, respectively,
can be normalised in such a way that~\eqref{bi} holds.
\item
If $m \not=0$ and $m a \in \Nat$ (exceptional situation),
we distinguish two cases:
\begin{enumerate}
\item
If $m(1+a)$ is odd (which necessarily implies that $m(1-a)$ is odd as well),
the eigenvalue~$\lambda$ is algebraically simple.
In view of~\eqref{norm.0.odd}, 
the functions~$\psi$ and~$\phi$ given 
by~\eqref{ef.0.odd} and~\eqref{ef*.0.generic}, respectively,
can be normalised in such a way that~\eqref{bi} holds.
\item
If $m(1+a)$ is even (which necessarily implies that $m(1-a)$ is even as well),
then~$\lambda$ belongs to the exceptional situation in the~$-1$ class too.
In particular, the formula~\eqref{efg*.-1.exceptional}
holds here after the replacement
$m \mapsto m \frac{1-a}{2}$. 
\end{enumerate}
\end{enumerate}
\end{itemize}

We summarise the results of this subsection in the following theorem.
\begin{Theorem}\label{Thm.minimal}
The eigenfunctions of~$H$ together with the generalised eigenfunctions
form a mutually biorthogonal pair in $\sii((-\half,\half))$.
Consequently, the eigenfunctions of~$H$ together with the generalised eigenfunctions
form a minimal complete set in $\sii((-\half,\half))$.
In particular,  
the eigenfunctions of~$H$ form a minimal complete set in $\sii((-\half,\half))$
if, and only if, $a \not\in \Rat$.
\end{Theorem}

An analogue of this theorem holds for the adjoint operator~$H^*$ as well.

\subsection{Conditional basis}\label{Sec.conditional}
Recall that $\{\psi_j\}_{j\in\Nat} \subset \mathcal{H}$
is a \emph{conditional (or Schauder) basis} 
in a Hilbert space~$\mathcal{H}$
if every $f \in \mathcal{H}$ has a unique expansion
in the vectors $\{\psi_j\}_{j\in\Nat}$, \ie, 
\begin{equation}\label{basis}
  \forall f \in \mathcal{H},
  \quad
  \exists! \{\alpha_j\}_{j\in\Nat} \subset \Com,
  \quad
  f = \sum_{j=0}^\infty \alpha_j \psi_j
  \,.
\end{equation}
The minimal completeness of $\{\psi_j\}_{j\in\Nat}$
is a necessary condition 
for $\{\psi_j\}_{j\in\Nat}$ to be a conditional basis.
By~\cite[Lem.~3.3.3]{Davies_2007} (see also \cite[Prop.~5]{KSTV}),
another necessary condition for $\{\psi_j\}_{j \in \Nat}$ 
being a conditional basis is that 
the norms of the one-dimensional projections
\begin{equation}\label{projection}
  P_j := \psi_j (\phi_j,\cdot)
\end{equation}
are uniformly bounded in~$j$.
Since 
$
  \|P_j\| = \|\psi_j\| \|\phi_j\|
$,
this check reduces to a computation of elementary trigonometric integrals
in our case.

\begin{itemize}
\item
\fbox{$-1$~class eigenvalues}
Let $\lambda = \big(\frac{4m}{1-a}\big)^2$ with $m \in \Nat^*$.
\begin{enumerate}
\item
If $m \frac{1+a}{1-a} \not\in \Nat$ (generic situation),
recalling~\eqref{ef.-1.generic}, \eqref{ef*.-1.generic}
and~\eqref{norm.-1.generic}, 
we define $P:=\psi(\phi,\cdot)$ and find
\begin{equation}\label{proj.-1.generic}
  \|P\| 
  = \frac{\sqrt{\frac{1}{8}\left[4\pi+\frac{1-a}{m}
  \sin\left(\frac{4 m \pi}{1-a}\right)\right]}}
  {\sqrt{\frac{\pi}{4} (1-a)} \,
  \left|\sin\left(m\pi \frac{1+a}{1-a}\right)\right|}
  \,.
\end{equation} 
\item
If $m \frac{1+a}{1-a} \in \Nat$ (exceptional situation),
recalling~\eqref{ef.-1.exceptional}, \eqref{ef*.-1.exceptional},
\eqref{efg.-1.exceptional}, \eqref{efg*.-1.exceptional},
\eqref{norm.-1.exceptional}, \eqref{norm.-1.generalised}
and~\eqref{norm*.-1.exceptional}, we define
$P_1:=\psi_1(\phi_1,\cdot)$, $P_2:=\psi_2(\eta,\cdot)$, $P_3:= \xi (\phi_2,\cdot)$,
and find
\begin{equation}\label{proj.-1.exceptional}
\begin{aligned}
  \|P_1\| 
  &= \frac{\sqrt{2}}{\sqrt{1-a}}
  \,,
  \\
  \|P_2\| 
  &=  
  \frac{\sqrt{15(1-a)+16 m^2 \pi^2 (1+a)}}
  {2\sqrt{3} \, \pi \sqrt{1+a} \, m}
  \,,
  \\
  \|P_3\| 
  &= \frac{\sqrt{64 m^2 \pi^2-36(1-a)^2}}{2\sqrt{6}\,\pi\sqrt{1+a}\,(1-a)\,m}
  \,.
\end{aligned}
\end{equation} 
\end{enumerate}
\item
\fbox{$+1$~class eigenvalues}
Let $\lambda = \big(\frac{4m}{1+a}\big)^2$ with $m \in \Nat^*$.
\begin{enumerate}
\item
If $m \frac{1-a}{1+a} \not\in \Nat$ (generic situation),
recalling~\eqref{ef.+1.generic}, \eqref{ef*.+1.generic}
and~\eqref{norm.+1.generic}, 
we define $P:=\psi(\phi,\cdot)$ and find
\begin{equation}\label{proj.+1.generic}
  \|P\| 
  = \frac{\sqrt{\frac{1}{8}\left[4\pi+\frac{1+a}{m}
  \sin\left(\frac{4 m \pi}{1+a}\right)\right]}}
  {\sqrt{\frac{\pi}{4} (1+a)} \,
  \left|\sin\left(m\pi \frac{1-a}{1+a}\right)\right|}
  \,.
\end{equation} 
\item
If $m \frac{1-a}{1+a} \in \Nat$ (exceptional situation),
then~$\lambda$ belongs to the exceptional situation in the~$-1$ class too.
Hence, the analysis is reduced to the preceding case.
\end{enumerate}
\item
\fbox{$0$~class eigenvalues}
Let $\lambda = (2m)^2$ with $m \in \Nat$.
\begin{enumerate}
\item 
If $m=0$, recalling~\eqref{ef.0.zero}, \eqref{ef*.0.zero}
and~\eqref{norm.0.zero}, 
we define $P:=\psi(\phi,\cdot)$ and find
\begin{equation}\label{proj.0.zero}
  \|P\| 
  = \sqrt{\frac{4}{3}}
  \,.
\end{equation} 
\item
If $m \not=0$ and $m a \not \in \Nat$ (generic situation),
recalling~\eqref{ef.0.generic}, \eqref{ef*.0.generic} 
and~\eqref{norm.0.generic},  
we define $P:=\psi(\phi,\cdot)$ and find
\begin{equation}\label{proj.0.generic}
  \|P\| 
  = \frac{\sqrt{2}}{\sqrt{1-\cos\big(m\pi(1+a)\big)}}
  \,.
\end{equation} 
\item
If $m \not=0$ and $m a \in \Nat$ (exceptional situation),
we distinguish two cases:
\begin{enumerate}
\item
If $m(1+a)$ is odd (which necessarily implies that $m(1-a)$ is odd as well),
recalling~\eqref{ef.0.odd}, \eqref{ef*.0.generic} 
and~\eqref{norm.0.odd},   
we define $P:=\psi(\phi,\cdot)$ and find
\begin{equation}\label{proj.0.odd}
  \|P\| 
  = 1
  \,.
\end{equation} 
\item
If $m(1+a)$ is even (which necessarily implies that $m(1-a)$ is even as well),
then~$\lambda$ belongs to the exceptional situation in the~$-1$ class too.
Hence, the analysis is reduced to the case studied above.
\end{enumerate}
\end{enumerate}
\end{itemize}

Now we are in a position to establish Theorem~\ref{Thm.basis.intro}
announced in the introduction. 
\begin{proof}[Proof of Theorem~\ref{Thm.basis.intro}]
If $a \in \Rat$, the eigenfunctions of~$H$ cannot form 
a conditional basis in $\sii((-\half,\half))$,
because they are not even minimal complete by Theorem~\ref{Thm.minimal}.
To disprove the basis property in the case $a \not\in \Rat$, 
we show that the spectral projections~\eqref{projection} 
are not uniformly bounded.
To this aim, we consider for instance~\eqref{proj.0.generic}.  
By Dirichlet's theorem on Diophantine approximation
of irrational numbers (see, \eg, \cite[Thm.~1A]{Schmidt}), 
there exist sequences of integers 
$(p_k, q_k) \in \Int \times \Nat^*$    
such that $|p_k|\to\infty$ and $q_k \to \infty$ as $k\to \infty$ and
$$
  \left| a - \frac{p_k}{q_k} \right| < \frac{1}{q_k^2}
$$
for every $k \in \Nat$.
Consequently, choosing $m:=2q_k$, we get
$$
  \cos\big(m\pi(1+a)\big) 
  = \cos\left(2q_k\pi\left(a-\frac{p_k}{q_k}\right)\right)
  \xrightarrow[k \to \infty]{} 1
  \,.
$$
Restricting to spectral projections~\eqref{proj.0.generic}
from the~$0$ class, we thus obtain 
$$
  \sup_{j \in \Nat} \|P_j\| 
  \geq \sup_{m \in \Nat^*} \frac{\sqrt{2}}{\sqrt{1-\cos\big(m\pi(1+a)\big)}}
  \geq \sup_{k \in \Nat^*} \frac{\sqrt{2}}{\sqrt{1-\cos\big(2q_k\pi(1+a)\big)}}
  = \infty 
  \,.
$$
This concludes the proof of the theorem.
\end{proof}
\begin{Remark}\label{Rem.open}
If $a \in \Rat$, 
it is still possible that the generalised eigensystem
(\ie~the collection of eigenfunctions and generalised eigenfunctions) 
is a conditional basis. 
We leave this question open here. 
Anyway, let us demonstrate that 
the projections~\eqref{projection},
where $\{\psi_j\}_{j \in \Nat}$ and $\{\phi_j\}_{j \in \Nat}$
denote the biorthogonal pair formed by the eigenfunctions
and generalised eigenfunctions of~$H$ and~$H^*$, respectively,
are uniformly bounded. 
The formulae~\eqref{proj.-1.exceptional}, \eqref{proj.0.zero} and~\eqref{proj.0.odd} 
are obviously uniformly bounded in $m \in \Nat^*$.
To show that it is the case for
the remaining norms of one-dimensional projections~\eqref{proj.-1.generic}, 
\eqref{proj.+1.generic} and~\eqref{proj.0.generic}, too,
it is enough to
write $a=\frac{p}{q}$ with some integers $(p,q) \in \Int\times\Int^*$
(since $|a|<1$, we have $|q| > |p|$)
and use the elementary estimates
$$
\begin{aligned}
  \left|\sin\left(m_{-1}\pi \frac{1+a}{1-a}\right)\right|
  &\geq \frac{2}{\pi} \,
  \dist\left(m_{-1}\pi \frac{1+a}{1-a},\pi \Int\right)
  \geq \frac{2}{|q-p|}
  \,,
  \\
  \left|\sin\left(m_{+1}\pi \frac{1-a}{1+a}\right)\right|
  &\geq \frac{2}{\pi} \,
  \dist \left(m_{+1}\pi \frac{1-a}{1+a},\pi n\right)
  \geq \frac{2}{|q+p|}
  \,,
  \\
  1-\cos\big(m_0\pi(1+a)\big) 
  &\geq \frac{4}{\pi^2} \, 
  \dist \Big(m_0\pi(1+a),2\pi \Int\Big)^2
  \geq \frac{4}{q^2}
  \,,
\end{aligned}
$$
valid for all
$m_{-1}, m_{+1}, m_0 \in \Nat^*$ 
such that $m_{\pm 1} \frac{1 \mp a}{1 \pm a} \not\in \Nat$ 
and $m_0 a \not\in \Nat$.
\end{Remark}

\subsection{Metric operator}\label{Sec.metric}
We finally recall that $\{\psi_j\}_{j \in \Nat}$, 
normalised to~$1$ in a Hilbert space $\mathcal{H}$, 
is an \emph{unconditional (or Riesz) basis}
if it is a conditional basis and the inequality 
\begin{equation}\label{KS-Riesz}
\forall f \in \mathcal{H}, \qquad
  C^{-1} \|f\|^2
  \leq \sum_{j=0}^\infty |(\psi_j,f)|^2 \leq 
  C \|f\|^2
\end{equation}
holds with a positive constant $C$ independent of $f$.
If $\{\psi_j\}_{j\in\Nat}$ is a normalised set of eigenfunctions
of an operator~$H$ with compact resolvent in~$\mathcal{H}$,
then~$H$ is similar to a normal operator 
via bounded and boundedly invertible transformation
if, and only if,
$\{\psi_j\}_{j\in\Nat}$ is an unconditional basis in~$\mathcal{H}$,
\cf~\cite[Thm.~3.4.5]{Davies_2007}.
The latter is equivalent to the similarity to 
a self-adjoint operator if the spectrum of~$H$ is in addition real.

The similarity to a self-adjoint operator
is also equivalent to the existence of a \emph{metric operator},
\ie~a positive, bounded and boundedly invertible operator~$\Theta$ 
such that~\eqref{quasi.intro} holds (\cf~\cite[Prop.~5.5.2]{KS-book}).
The metric operator can be constructed by the formula 
\begin{equation}\label{metric}
  \Theta = \sum_{j=0}^\infty \phi_j (\phi_j,\cdot)
  \,,
\end{equation}
where~$\phi_j$ are eigenfunctions of~$H^*$. 

In our case, $H$~cannot be similar to a self-adjoint operator
via bounded and boundedly invertible transformation
because the eigenfunctions of~$H$ do not form already a conditional basis
(they are not even complete if $a \in \Rat$),
\cf~Theorem~\ref{Thm.basis.intro}.
Nonetheless, if $a \not\in \Rat$,
we shall show that the relation~\eqref{quasi.intro} still holds
with a positive and bounded~$\Theta$ whose inverse exists 
but it is unbounded.  
Furthermore, we shall derive a closed formula
for the metric operator~\eqref{metric}.

Our approach is based on the following peculiar properties
of the eigenbasis of~$H^*$. 
Hereafter we assume $a \not\in \Rat$.

\begin{itemize}
\item
Eigenfunctions in the $-1$ class are all those eigenfunctions
of the Dirichlet Laplacian in $(\half a,\half)$ 
which are antisymmetric with respect to 
the middle point $\frac{\pi}{4}(1+a)$.
Putting $A_+ := \sqrt{2/[\pi(1-a)]}$,
the eigenfunctions become normalised to~$1$ in $\sii((\half a,\half))$.
Consequently,
$$
  \sum_{\lambda_j \in \sigma_-} \phi_j (\phi_j,\cdot)
  = 0 \oplus P_+ 
  \,,
$$
where~$P_+$ is the antisymmetric projection 
$$
  (P_+ f)(x) := \frac{f(x)-f(-x+\half(1+a))}{2}
  \,, \qquad 
  x \in [\half a,\half]
  \,.
$$
The direct sum is again with respect to the decomposition
$\sii((-\half,\half a)) \oplus \sii((\half a,\half))$.
\item
Eigenfunctions in the $+1$ class are all those eigenfunctions
of the Dirichlet Laplacian in $(-\half,\half a)$ 
which are antisymmetric with respect to 
the middle point $-\frac{\pi}{4}(1-a)$.
Putting $A_- := \sqrt{2/[\pi(1+a)]}$,
the eigenfunctions become normalised to~$1$ in $\sii((-\half,\half a))$.
Consequently,  
$$
  \sum_{\lambda_j \in \sigma_+} \phi_j (\phi_j,\cdot)
  = P_- \oplus 0
  \,,
$$
where~$P_-$ is the antisymmetric projection 
$$
  (P_- f)(x) := \frac{f(x)-f(-x-\half(1-a))}{2}
  \,, \qquad 
  x \in [-\half,\half a]
  \,.
$$
\item
Eigenfunctions in the $0$ class except for~\eqref{ef*.0.zero}
are all those eigenfunctions
of the Dirichlet Laplacian in $(-\half,\half)$ 
which are antisymmetric with respect to 
the middle point~$0$.
Putting $C := \sqrt{2/\pi}$,
the eigenfunctions become normalised to~$1$ in $\sii((-\half,\half))$.
Consequently,
$$
  \sum_{\lambda_j \in \sigma_0\setminus\{0\}} \phi_j (\phi_j,\cdot)
  = P_0 
  \,,
$$
where~$P_0$ is the antisymmetric projection 
$$
  (P_0 f)(x) := \frac{f(x)-f(-x)}{2}
  \,, \qquad 
  x \in [-\half,\half]
  \,.
$$
\item
Finally, let us denote the eigenfunction~\eqref{ef*.0.zero} 
corresponding to the zero eigenvalue by~$\phi_0$
and let us put the normalisation constant~$C$ equal to one for instance. 
Then we get a rank-one operator
$$
  \sum_{\lambda_j = 0} \phi_j (\phi_j,\cdot)
  = \phi_0 (\phi_0,\cdot)
  \,.
$$
\end{itemize}

Summing up, we arrive at the following particularly simple
form for the metric operator defined by~\eqref{metric}
\begin{equation}\label{metric.ours}
  \Theta = \phi_0 (\phi_0,\cdot) + P_0 + P_- \oplus P_+ 
  \,.
\end{equation}

Let us carefully verify all the required properties 
of the metric operator,
giving thus a proof Theorem~\ref{Thm.metric.intro}
announced in the introduction.
\begin{proof}[Proof of Theorem~\ref{Thm.metric.intro}] \
\begin{itemize}
\item
Obviously, $\Theta$ defined by~\eqref{metric.ours} is \textbf{bounded}.
\item
It is \textbf{positive} just because
\begin{equation}\label{positive}
  (f,\Theta f) =
  |(\phi_0,f)|^2 + \|P_0 f\|^2 + \|P_- f \oplus P_+ f\|^2
  \geq 0
\end{equation}
for every $f \in \sii((-\half,\half))$.
\item
To prove that~$\Theta$ is \textbf{invertible}
(\ie~$0$ is not an eigenvalue of~$\Theta$),
we need the following fact.
\begin{Lemma}
Let $a \not\in \Rat$. 
If $P_0 f=0$ and $P_- f \oplus P_+ f=0$ for some $f \in \sii((-\half,\half))$,
then $f(x)$ is a constant for almost every $x\in(-\half,\half)$.
\end{Lemma}
\begin{proof}
We decompose~$f$ into the eigenbasis of the Neumann Laplacian in $(-\half,\half)$,
\ie, we write
$$
  f = \sum_{n=0}^\infty \alpha_n \chi_n
  \,, \qquad
  \chi_n(x) := 
  \begin{cases}
    \sqrt{\frac{2}{\pi}} \cos(nx)
    & \mbox{if} \quad n \geq 1 \ \mbox{is even} \,,
    \\
    \sqrt{\frac{2}{\pi}} \sin(nx)
    & \mbox{if} \quad n \geq 1 \ \mbox{is odd} \,,
    \\
    \sqrt{\frac{1}{\pi}}
    & \mbox{if} \quad n = 0 \,,
  \end{cases}
$$ 
where $\alpha_n := (\chi_n,f)$.
Requiring $P_0 f=0$ immediately yields that the coefficients~$\alpha_n$
vanish for all odd~$n$. 
At the same time, an explicit computation gives
$$
  (\chi_m,P_- \chi_n \oplus P_+ \chi_n) = 
  \frac{1}{2} \left[1-\cos(\frac{n\pi}{2}) \cos(\frac{n\pi a}{2})\right]
  \, \delta_{mn}
$$
for all even $m,n$. Summing up,
$$
  \|P_0 f\|^2 + \|P_- f \oplus P_+ f\|^2
  = \sum_{n \ \textrm{odd}} |\alpha_n|^2
  + \sum_{n \ \textrm{even}} |\alpha_n|^2 
  \frac{1}{2} \left[1-\cos(\frac{n\pi}{2}) \cos(\frac{n\pi a}{2})\right]
  \,.
$$
If $a \not\in \Rat$, the square bracket is positive for all $n \not= 0$
and we may conclude that $\alpha_n=0$ for all $n \geq 1$.
Consequently, $f(x) = \alpha_0 \chi_0(x)$ 
for almost every $x\in(-\half,\half)$.
\end{proof}
\noindent
Using this lemma, 
assuming that $f \not= 0$ is an eigenfunction of~$\Theta$
corresponding to its zero eigenvalue,
we conclude from~\eqref{positive} that $f(x) = \const \in\Com$ 
for almost every $x \in (-\half,\half)$
and
$$
  0 = (\phi_0,\psi) = \const \;\!  
  \left(\frac{\pi}{2}\right)^2 (a^2-1)
  \,,
$$
which can be satisfied only if $\const = 0$, a contradiction.
Hence~$\Theta$ is invertible.
\item
Recall that~$\Theta$ is \textbf{not boundedly invertible}
(\ie~$0$ is in the continuous spectrum of~$\Theta$),
otherwise the eigenfunctions of~$H$ would form an unconditional basis,
which contradicts Theorem~\ref{Thm.basis.intro}.
\item
Finally, let us show that 
the \textbf{quasi-self-adjointness} relation~\eqref{quasi.intro} holds. 

First of all,
we have to check that~$\Theta$ properly maps $\Dom(H)$ to $\Dom(H^*)$.
It is obvious for the first term $\phi_0(\phi_0,\cdot)$ in~\eqref{metric.ours}.
Let $\psi \in \Dom(H)$. 
We clearly have 
$$
  P_0 H^2((-\half,\half)) = H^2((-\half,\half))
  \,, \qquad
  (P_-\oplus P_+) H^2((-\half,\half)) 
  = H^2((-\half,\half a)) \oplus H^2((\half a,\half))
  \,.
$$
Using the antisymmetric nature of the projections~$P_0$, $P_\pm$ 
and the boundary conditions $f \in \Dom(H)$ satisfies, 
we easily find
\begin{align*}
  (P_- f)(-\half) &=0 \,,
  &(P_+ f)(\half) &=0 \,,
  &(P_0 f)(\pm\half) &=0 \,,
  \\
  (P_- f)(\half a-) &=0 \,,
  &(P_+ f)(\half a+) &=0 \,,
  &(P_0 f)(\half a \pm) &=\frac{f(\half a-)-f(-\half a)}{2} \,,
\end{align*}
and
\begin{align*}
  (P_0 f)'(\half) - (P_0 f)'(-\half) &=0 \,,
  \\
  (P_0 f)'(\half a+) - (P_0 f)'(\half a-) &=0 \,,
  \\
  (P_- f \oplus P_+ f)'(\half) - (P_- f \oplus P_+ f)'(-\half)
  &= \frac{f'(\half)-f'(-\half)}{2} \,,
  \\
  (P_- f \oplus P_+ f)'(\half a+) 
  - (P_- f \oplus P_+ f)'(\half a-)
  &= \frac{f'(\half)-f'(-\half)}{2} \,.
\end{align*}
Hence $\Theta f \in \Dom(H^*)$.

Verifying the identity $(f\psi)''(x)=(\Theta f'')(x)$
for $x \in (-\half,\half a) \cup (\half a,\half)$ 
is straightforward.
\end{itemize}
This concludes the proof of Theorem~\ref{Thm.metric.intro}. 
\end{proof}
%

\section{Some open problems}\label{Sec.end}
%
Let us conclude this paper by suggesting some further research questions
related to problems of the type~\eqref{genproblem}. 
The list is certainly not complete and we just added those questions 
which are most directly connected with our present contribution.
\begin{itemize}
\item
If $a \in \Rat$,
do the eigenfunctions together with the generalised eigenfunctions
form a conditional basis (\cf~Remark~\ref{Rem.open})?
\item 
Is there a direct operator-theoretic argument for the fact 
that the spectrum of the operator
associated with~\eqref{genproblem} is always real? 
This has been shown in \cite{Leung-Li-Rakesh-2008} 
using results about the zero set of trigonometric series.
\item 
Is it possible to derive related results about the spectrum 
and the multiplicity for more general jump distributions 
than those considered in the present work?
\item 
If one replaces the operator $-\frac{d^2}{dx^2}$ 
by $-\frac{\sigma^2}{2}\frac{d^2}{dx^2} - b\frac{d}{dx}$ 
in $(-\half,\half)$, 
then it is shown probabilistically partially in~\cite{Kolb-Wuebker-2011b} 
and fully in~\cite{Ben-Ari_2014} that the spectral gap, 
denoted by $\gamma_1(\sigma,b)$, 
of the corresponding diffusion with jump distribution~$\delta_0$ is given by
$$
  \gamma_1(\sigma,b) = 
  \min\left\{
  \lambda_0^{(0,\frac{\pi}{2})}(\sigma,b)
  ,
  \lambda_0^{(0,\frac{\pi}{4})}(\sigma,0)
  \right\}
  .
$$
Here we denote by  $\lambda_0^{(0,l)}(\sigma,b)$ 
the smallest Dirichlet eigenvalue of 
$-\frac{\sigma^2}{2}\frac{d^2}{dx^2} - b\frac{d}{dx}$ in the interval $(0,l)$.  
Thus
\begin{displaymath}
\gamma_1(\sigma,\mu) = \begin{cases}
{2\sigma^2}+\frac{b^2}{2\sigma^2} &\text{if \ $|b| \leq 2 \sqrt{3}\,\sigma^2$}
  \,,
  \\
8\sigma^2 &\text{otherwise}
  \,.
\end{cases}
\end{displaymath}
In particular, the spectral gap stays constant once~$|b|$ 
is greater than $2\sqrt{3}{\sigma^2}$. 
An investigation of the full spectrum including multiplicities 
and its dependence on the drift $b$ might reveal further interesting properties. 
\end{itemize}

Finally, let us mention that the stochastic process 
described in~\eqref{genproblem}
is still not fully understood probabilistically;
for recent developments we refer to~\cite{Ben-Ari-Panzo-Tripp}.

\subsection*{Acknowledgment}
%
The initial research published in 2016 was supported
by the project RVO61389005 and the GACR grant No.\ 14-06818S.
The corrigendum published in 2019 was supported 
by the GACR grant No.\ 18-08835S.
We are grateful to our colleagues 
Lyonell Boulton,Vladimir Lotoreichik, 
Konstantin Pankrashkin and Mat\v{e}j Tu\v{s}ek
for useful discussions.

%

\providecommand{\bysame}{\leavevmode\hbox to3em{\hrulefill}\thinspace}
\providecommand{\MR}{\relax\ifhmode\unskip\space\fi MR }
\providecommand{\MRhref}[2]{%
  \href{http://www.ams.org/mathscinet-getitem?mr=#1}{#2}
}
\providecommand{\href}[2]{#2}

\end{document}